\documentclass[a4paper;11pt]{amsart}
\usepackage[noadjust]{cite}\usepackage{mathrsfs,cleveref}
\usepackage{tikz}
\usepackage{mathrsfs}
 \usepackage[arrow,matrix]{xy}
\usepackage{amsmath,amssymb,amscd,bbm,amsthm,  mathrsfs}
\usepackage{amsfonts,enumerate}
\newtheorem{thm}{Theorem}[section]
\newtheorem{lem}{Lemma}[section]
\newtheorem{cor}{Corollary}[section]
\newtheorem{prop}{Proposition}[section]
\newtheorem{rem}{Remark}[section]

\theoremstyle{definition}

\begin{document}

\title[On  monogenic  functions and the Dirac complex of two vector variables]{On   monogenic  functions and the Dirac complex of two vector variables}
%----------Author 1
\author[Y. Shi]{Yun Shi}
%\author[]{Rafa\l \ Ab\l amowicz}
\address{%
Department of Mathematics\\ Zhejiang University of Science and Technology\\ Hangzhou 310023, China
 }
\email{shiyun@zust.edu.cn}
%

%----------Author 2
\author[W. Wang]{Wei Wang}
%\author[]{Rafa\l \ Ab\l amowicz}
\address{%
Department of Mathematics\\ Zhejiang University\\ Hangzhou 310027, China
 }
\email{wwang@zju.edu.cn}
\author[Q. Wu]{Qingyan Wu}
%\author[]{Rafa\l \ Ab\l amowicz}
\address{%
Department of Mathematics\\ Linyi University\\ Linyi  276005, China
 }
\email{wuqingyan@lyu.edu.cn}
\thanks{%The first author is  partially supported by  Nature Science Foundation of Zhejiang province (No. LY22A010013) and  National Nature Science Foundation in China (Nos. 11801508, 11971425); The second author is partially supported by National Nature Science Foundation in China  (No. 11971425); The third author is partially supported by National Nature Science Foundation in China  (No. 12071197), the Natural ScienceFoundation of Shandong Province (Nos. ZR2019YQ04, 2020KJI002).
}
\date{\today}

%%% ----------------------------------------------------------------------
\begin{abstract}
A monogenic function of two vector variables is a function annihilated by the operator consisting of two Dirac operators, which are  associated to two variables, respectively. We give the explicit form of   differential operators in the Dirac complex resolving this operator and prove its ellipticity directly. This open the door to apply the method of several complex variables to investigate this kind of monogenic functions.  We prove the Poincar\'e lemma for this complex, i.e. the non-homogeneous  equations  are solvable under the compatibility condition  by solving  the associated Hodge   Laplacian equations of fourth order.  As   corollaries, we   establish  the  Bochner--Martinelli integral representation formula for this differential operator  and   the Hartogs' extension phenomenon for   monogenic functions. We also apply      abstract duality theorem to the Dirac  complex to obtain the generalization of Malgrange's vanishing theorem and establish  the Hartogs--Bochner extension phenomenon for monogenic  functions under the  moment  condition.
\end{abstract}

%%% ----------------------------------------------------------------------
\maketitle
%%% ----------------------------------------------------------------------
%\tableofcontents
\section{Introduction}
Since Pertici \cite{Pertici} proved Hartogs' phenomenon for regular functions of several quaternionic variables, analysis of this kind of regular functions developed rapidly.  The method for this analysis is to solve   non-homogeneous Cauchy-Fueter equation, which is  overdetermined. Therefore  it is necessary to find its compatibility condition, more generally,    the resolution of the  Cauchy--Fueter operator. The search of this complex began in 1990s by using computer method (cf. e.g. \cite{Adams1,Adams3,CSSS} and references therein). Later, it was  realized that  the Penrose transformation can be applied to construct resolutions after complexification (cf. \cite{MR1165872,bures1,Bures2,CSS,SW,Wa10} and reference therein). Both method can be applied to  monogenic  functions of several vector variables, which are annihilated by several Dirac operators, and the construction of their resolutions \cite{Damiano,Krump,kru,kru2,Sabadini,Salac}.  %Wang \cite{wang08,Wa10} solve the system of non-homogeneous Cauchy--Fueter equations and $k$-Cauchy--Fueter equations  prove the Hartogs' extension phenomenon for quaternionic regular functions and $k$-regular functions  on any domain, respectively. Recently, Ren and Zhang \cite{RZ} show the failure of Hartogs phenomena for discrete $k$-Cauchy--Fueter operators and show that the Hartogs theorem remains valid for the pair $K\subset\Omega$ if $K$ is convex and the distance between $K$ and the complement set of  is larger than $4.$
On the other hand,  Ren and H.-Y. Wang \cite{RW} have already solved the non-homogeneous several Dirac equation with a compatibility condition in the integral form, proved the Bochner--Martinelli formula and Hartogs' phenomenon for monogenic functions of several vector variables by generalizing method used by Pertici.

Let $\mathbb R_n$ be the real Clifford algebra. The   \emph{Dirac operator} acts on functions  $f:\mathbb R^n\rightarrow \mathbb R_n$ as the linear operator
\begin{align}
\partial_{\mathbf x}=\sum_{i=1}^ne_i{\partial_{x_{i}}},
\end{align}
where $\mathbf x=\left(x_1,\cdots, x_n\right)\in\mathbb R^n,$  and $e_i,i=1,\cdots, n,$ is    Clifford number. The equation $\partial_{\mathbf x} f = g$ has a smooth solution $f$ for  smooth  $g$   on  suitable open sets (cf. e.g. \cite{Brackx,Delanghe}). Now consider the Cartesian product $\mathbb R^{kn}.$ If we write  $\mathbf x_A=\left(x_{A1},\cdots,x_{An}\right)$ as  the vector variable in the $A$-th copy of $\mathbb R^n,$ $\partial_{Ai}:=\frac{\partial}{\partial{x_{Ai}}},A=0,\cdots,k-1,i=1, \cdots,n,$ for simplicity. The \emph{several Dirac operator} is
\begin{align}
\partial_{\mathbf x_A}=\sum_{i=1}^ne_i\partial_{Ai},
\end{align}
acting on functions $f:\Omega\subset \mathbb R^{kn}\rightarrow\mathbb R_n.$ $f$ is called \emph{monogenic} on $\Omega$ if it satisfies
\begin{align}\label{pf0}
\partial_{\mathbf x_A}f=0,
\end{align}
on $\Omega$ for $A=0,\cdots,k-1.$

As in the complex and quaternionic case,  to investigate this kind of monogenic functions, a fundamental method is to solve the associated inhomogeneous system
\begin{equation}\begin{aligned}\label{pf}
\left\{
\begin{array}{rcl}\partial_{\mathbf x_0}f=&g_0,\\\partial_{\mathbf x_1}f=&g_1,\\\vdots&\\\partial_{\mathbf x_{k-1}}f=&g_{k-1},
\end{array}\right.
\end{aligned}\end{equation}
where $f$ and $g_i$ are in a suitable space of generalized functions. (\ref{pf}) can only be solved under a   compatibility condition since it is overdetermined. Thus the solution of this problem can be obtained if one can provide a description of the so-called Dirac complex.

It is more simple to consider function $f$ in (\ref{pf0}) valued in an irreducible module of ${\rm Spin}(n),$ in particular, the spinor modules $\mathbb S^{\pm}.$ In this paper  $\mathbb S^\pm$  denote the two spinor modules for $n$ even, and the same symbols are used for $n$ odd with the convention that $\mathbb S^+$ and $\mathbb S^-$ are isomorphic.

For two vector variables, this  complex is explicitly known by Damiano-Sabadini-Sou\v cek \cite{Damiano} as
\begin{equation}\begin{aligned}\label{co}
0\rightarrow\Gamma\left(\Omega,\mathscr V_0 \right)\xrightarrow{\mathscr{D}_{0}} \Gamma\left(\Omega,\mathscr V_1 \right)\xrightarrow{\mathscr{D}_{1}}\Gamma\left(\Omega,\mathscr V_2\right) \xrightarrow{\mathscr{D}_{2}} \Gamma\left(\Omega,\mathscr V_3\right)\rightarrow 0,
\end{aligned}\end{equation}
where $\Omega$ is a domain in $\mathbb R^{2n},$ and  \begin{align*}\mathscr V_0=\mathbb V_{00}^+,\quad\mathscr V_1=\mathbb V_{10}^-,\quad\mathscr V_2=\mathbb V_{21}^-,\quad\mathscr V_3=\mathbb V_{22}^+.\end{align*}
Here $\mathbb V_\lambda^\pm=V_\lambda\otimes \mathbb S^\pm,$ and $V_\lambda$ is an irreducible ${\rm GL}(2)$-module with the highest $\lambda.$ %or just the Young diagram corresponding to the partition.  %In the following of the paper we write $\mathscr V_0=\mathbb V_{0}^+,\mathscr V_1=\mathbb V_{1}^-,\mathscr V_2=\mathbb V_{21}^-,\mathscr V_3=\mathbb V_{22}^+$ for simplicity.
It is known that $$ V_{00}\cong\mathbb C,\quad V_{10}\cong\mathbb C^2,\quad V_{21}\cong\mathbb C^2,\quad V_{22}\cong\mathbb C,$$  are    representations of ${\rm GL}(2)$ \cite{FH}. As complex vector spaces, we have, \begin{equation}\begin{aligned}\mathscr V_0=\mathbb S^+,\ \mathscr V_1=\mathbb C^2\otimes\mathbb S^-,\ \mathscr V_2=\mathbb C^2\otimes\mathbb S^+,\ \mathscr V_3=\mathbb S^-.
\end{aligned}\end{equation}

Although operators $\mathscr D_j$ in the complex (\ref{co}) are known to be invariant operators, but their explicit form were not given in \cite{Damiano}. %The first operator $\mathscr D_0$ in the complex (\ref{co})   maps functions on $\left(\mathbb R^n\right)^2$ with values in $\mathbb S^+$ to functions on $\left(\mathbb R^n\right)^2$ with value in $\mathbb V_1\times\mathbb S^-,$ where $\mathbb V_1\cong\mathbb C^2$ is the standard  representation of ${\rm GL(2)}.$ It can be described as the map $\mathbb V_0^+\rightarrow \mathbb V_1^-,$   which acting on functions $f : \Gamma\left(\left(\mathbb R^n\right)^2,\mathbb S^\pm\right)\rightarrow \Gamma\left(\left(\mathbb R^n\right)^2,\mathbb S^\mp\right),$ as
Denote the linear operator
\begin{align}\label{nabla}
\nabla_A:=\sum_{j=1}^n\gamma_j\partial_{A j},
\end{align}
for $\mathbf x_A=\left(x_{A 1},\cdots,x_{A n}\right),$ where $\gamma_j:\mathbb S^\pm\rightarrow\mathbb S^\mp$ are Dirac matrices. A section in $\Gamma\left(\Omega,\mathscr V_0\right)$ is an   $\mathbb S^+$-valued function on $\Omega$, while a section in $\Gamma\left(\Omega,\mathscr V_1\right)$ is written as   $g=\left(\begin{matrix}  g_0\\g_1\end{matrix}\right)$ for some   $\mathbb S^-$-valued functions $g_0$ and $g_1$ on $\Omega.$  Similarly, a section in $\Gamma\left(\Omega,\mathscr V_2\right)$ is also written as  $h=\left(\begin{matrix}  h_0\\h_1\end{matrix}\right)$ for some  $\mathbb S^+$-valued function $h_0$ and $h_1$ on $\Omega.$    The  explicit form of operators $\mathscr D_j$'s are as follows
\begin{equation}\begin{aligned}\label{D}
\left(\mathscr D_0f\right)_A:=&\nabla_A f,\quad &{\rm for}\ f\in\Gamma\left(\Omega,\mathscr V_0\right),\\\left(\mathscr D_1g\right)_A:=&\nabla_0\nabla_A g_1-\nabla_1\nabla_A g_0, \quad &{\rm for}\ g\in\Gamma\left(\Omega,\mathscr V_1\right), \\\mathscr D_2h:=&2\nabla_{[0}h_{1]}:=\nabla_0h_1-\nabla_1h_0,\quad &{\rm for}\ h\in\Gamma\left(\Omega,\mathscr V_2\right),
\end{aligned}\end{equation}
where $A=0,1.$

The Dirac complex on $\mathbb R^{kn}$ for $k=3,n\geq6$ is given in \cite{Damiano}, but their operators are more complicated than the case of $k=2$. The case of  $k=4$ is discussed by Krump in \cite{kru}.  For the stable range $n\geq 2k,$ the Dirac complex is known \cite{Goodmann}. But for the unstable range, it is an open problem to construct the Dirac complex. However  some  results   can be found in \cite{kru2}.

We prove  the ellipticity of this differential complex (\ref{co}) directly. As in the case of several quaternionic variables (cf. e.g. \cite{LW,RZ,SW2,wang08,wang222,wang29}), once we know  differential complex explicitly, which is also  elliptic, the method of several complex variables can be applied to obtain many profound results about monogenic functions. We define the associated Hodge Laplacian operators of fourth order: \begin{equation}\begin{aligned}\label{hodge}
\Box_0:=&\left(\mathscr D_0^*\mathscr D_0\right)^2,\\\Box_1:=&\left(\mathscr D_0\mathscr D_0^*\right)^2+\mathscr D_1^*\mathscr D_1,\\\Box_2:=&\mathscr D_1\mathscr D_1^*+\left(\mathscr D_2^*\mathscr D_2\right)^2,
\end{aligned}\end{equation} which are elliptic operators on $\Gamma\left(\Omega,\mathscr V_j\right),$ where $\mathscr D_j^*$ is the formal  adjoint of $\mathscr D_j.$ These operators have very simply form
\begin{align}\label{box22}\Box_0=\Delta^2,\quad \Box_1=\Box_2=\left(\begin{array}{cc} \Delta^2&\\  & \Delta^2\end{array}\right),
\end{align}
where $\Delta=\sum_{j=1}^n\left(\partial_{0j}^2+\partial_{1j}^2\right)$ is the Laplacian operator on $\mathbb R^{2n}.$ So their fundamental solution are given by   $\frac{1}{|\mathbf x|^{2n-4}}$ up to a constant. This allows us to derive  Bochner--Martinelli formula directly in a very simple way and  solve the non-homogeneous Dirac equation
\begin{align}\label{duf}
\mathscr D_ju=f,
\end{align}
 under the compatibility  condition
\begin{align}\label{comp}
\mathscr D_{j+1}f=0.
\end{align}
  %Until now, a complete answer to this question,   for all positive integers $k$ and $n,$ is not known \cite{Damiano}.
\begin{thm}\label{t31}
Suppose that $f\in L^2\left(\mathbb R^{2n},\mathscr V_j\right)$ satisfies the compatibility condition {\rm(\ref{comp})} in sense of distributions. Then there exists a function $u\in W^{1,2}\left(\mathbb R^{2n},\mathscr V_j\right)$ satisfying the non-homogeneous   equation {\rm(\ref{duf})}. Furthermore, if $f\in C_0\left(\mathbb R^{2n},\mathscr V_0\right)$ with $\mathscr D_1f=0$ in the sense of distributions,  then there exists a function $u\in C_0\left(\mathbb R^{2n},\mathscr V_0\right)\cap W^{1,2}\left(\mathbb R^{2n},\mathscr V_0\right)$ satisfying {\rm(\ref{duf})} and vansihing on the unbounded connected component of $\mathbb R^{2n}\setminus {\rm supp}f.$
\end{thm}
As a corollary, we can  prove the following  Hartogs' phenomenon.
A function $u:\Omega\rightarrow\mathscr V_0$ is  called \emph{monogenic} on $\Omega$ if \begin{align*}
\mathscr D_0u(\mathbf x)=0,
\end{align*}
for any $\mathbf x\in \Omega.$ The space of all monogenic  functions on $\Omega$ is denoted by $\mathcal O(\Omega).$
\begin{thm}\label{hartogs}
Let $\Omega$ be an open set in $\mathbb R^{2n}$ and let $K$ be a compact subset of $\Omega$ such that $\Omega\setminus K$ is connected. Then for each $u\in \mathcal O(\Omega\setminus K),$ we can find $U\in\mathcal O(\Omega)$ such that $U=u$ in $\Omega\setminus K.$
\end{thm}
%Abstract duality theorem for a Fr\'echet--Schwartz space or the dual of a Fr\'echet--Schwartz space with topological homomorphisms can be applied to our case.
For  the Dirac complex ${\mathscr V}_\bullet,$   let $\mathcal E\left(\mathbb R^{2n}, {\mathscr V}_\bullet\right)$ be the space of smooth ${\mathscr V}_\bullet$-valued functions with the topology of uniform convergence on compact sets of the functions and all their derivatives. Let $\mathcal D\left(\mathbb R^{2n}, {\mathscr V}_\bullet\right)$ be the space of compactly supported elements of $\mathcal E\left(\mathbb R^{2n}, {\mathscr V}_\bullet\right).$ Denote $\mathcal E'\left(\mathbb R^{2n}, {\mathscr V}_\bullet\right),$ the dual of $\mathcal E\left(\mathbb R^{2n}, {\mathscr V}_\bullet\right).$ By using abstract duality theorem, we have the following generalization of Malgrange's vanishing theorem.
\begin{thm}\label{vanish}
The cohomology groups $H^3\left(\mathcal E\left(\mathbb R^{2n},{\mathscr V}_\bullet\right)\right)$ and $H^3\left(\mathcal D'\left(\mathbb R^{2n},{\mathscr V}_\bullet\right)\right)$ both vanish.
\end{thm}

Moreover, we give the Hartogs--Bochner extension for monogenic functions under the moment condition.
\begin{thm}\label{text}
Let $\Omega$ be a domain of $\mathbb R^{2n}$ with smooth boundary such that $\mathbb R^{2n}\setminus \overline\Omega$ connected. %and let $\rho$ be a defining function {\rm (}i.e. $\rho = 0$ on $\partial\Omega$	 and $\rho < 0$ in $\Omega${\rm )} such that $|{\rm grad}\ \rho| = 1.$
Suppose that $f$ is the restriction to $\partial\Omega$	 of a $C^2\left(\Omega, \mathscr V_0\right)$ function, with $\mathscr D_0 f$ vanishing to the second order on $\partial \Omega,$ and satisfies the moment condition
\begin{align}\label{moment}
\int_{\partial \Omega}\left\langle f,n_0G_0+n_1G_1\right\rangle_{\mathbb S^+} {{\rm d}S}=0,
\end{align}
for any $G=\left(\begin{matrix}G_0\\G_1\end{matrix}\right)\in\ker {\mathscr D}^*_0\subset\mathcal E\left(\mathbb R^{2n}, {\mathscr V}_1\right),$ ${\mathscr D}^*_0G=0$ on $\overline\Omega,$ where $n_A=\sum_{j=1}^n\gamma_jn_{Aj}$ for $A=0,1,$ and $n=\left(n_{01},\cdots,n_{0n},n_{11},\cdots,n_{1n}\right)$ is the unit outer normal vector  to $\partial\Omega.$ Then there exists a  monogenic  function $\tilde f\in \mathcal O(\Omega)$  such that $\tilde f=f$ on $\partial \Omega.$
\end{thm}
Here the  vanishing of  $\mathscr D_0 f$ to the second order on $\partial \Omega$ and the moment condition  (\ref{moment}) are also the  necessary condition for the Hartogs--Bochner extension (see Remark \ref{r4.7}).

  %In this paper we give the  explicit form of operators $\mathscr D_j$ in the so-called Dirac complex  (\ref{D}) with $k=2,$ which allow us to defined the associated Hodge Laplacian operators (\ref{hodge}), and check their ellipticity. It help us to prove the Bochner--Martinelli formula, solve the non-homogeneous Cauchy-Riemann equation easily. Moreover, we can not only prove the Hartogs' phenomenon but also the  Hartogs--Bochner extension for monogenic functions.

%The Dirac operator in several vector variables appeared already in Damiano, Krump, Sabadini, Sala\v c and Sou\v cek, et al (cf. eg. \cite{Damiano,Sabadini,Salac} and reference therein).

%For $k = 2,$ the full resolution has the form \begin{equation}\begin{aligned}\label{co} \mathscr V_0\xrightarrow{\mathscr{D}_{0}} \mathscr V_1\xrightarrow{\mathscr{D}_{1}} \mathbb \mathscr V_2 \xrightarrow{\mathscr{D}_{2}} \mathscr V_3, \end{aligned}\end{equation} where the differential operators $\mathscr D_j$ is given in (\ref{D}).

The paper is organized as follows. In Sect. 2, we give the preliminaries on Clifford algebras, the Dirac operator, spin modules  and  the differential  complex of two vector variables. Then we write down operators of the Dirac complex of two vector variables and prove its ellipticity.  In Sect. 3, we give the  Bochner--Martinelli integral representation formula for $\mathscr D_0$ by using    the fundamental solution to the Hodge  Laplacian $\Box_j.$ As applications, we   prove Theorem \ref{t31} and Theorem \ref{hartogs}. In Sect. 4,    we apply abstract duality theorem to the Dirac complex on $\mathbb R^{2n}$ to obtain the generalization of Malgrange's vanishing theorem and the Hartogs--Bochner extension for monogenic functions under the moment condition.

\section{The Dirac complex of two vector variables}
\subsection{Clifford algebra and spin modules $\mathbb S^\pm$}\label{s21}
The {\it real Clifford algebra} $\mathbb R_n$ is the associative algebra generated by  the $n$ basis elements of $\mathbb R^n$ satisfying
\begin{align*}
e_ie_j+e_je_i=-2\delta_{ij},
\end{align*}
for $i,j=1,\cdots,n.$
The basis of $\mathbb R_n$ is
\begin{align*}
e_0=1,\quad e_\alpha=e_{a_1}\cdots e_{a_i},
\end{align*}
for $1\leq i\leq n,1\leq a_1<\cdots<a_i\leq n.$
%For $n=2m,$ write $\mathbb C^{2m}=W\oplus W',$ where $W$ and $W'$ are $m$-dimensional isotropic subspaces. For $n=2m+1,$  write $\mathbb C^{2m+1}=W\oplus W'\oplus U,$ where $W$ and $W'$ are $m$-dimensional isotropic subspaces, and $U$ is a $1$-dimensional space perpendicular to them.
For $n=2m,$ $$W={\rm span}_{\mathbb C}\left\{f_j,j=1,\cdots,m,\right\},\quad W'={\rm span}_{\mathbb C}\left\{\bar f_{j},j=1,\cdots,m,\right\},$$   are maximally isotropic subspaces of \emph{complex Clifford algebra}  $\mathbb C_n$ and $\mathbb C^n=W\oplus W',$ where
\begin{align*}
f_j=\frac12\left(e_{2j-1}+\mathbf ie_{2j}\right), \ \bar f_j=-\frac12\left(e_{2j-1}-\mathbf ie_{2j}\right),
\end{align*}(cf.  \cite[(4.11)]{Delanghe}).
We have two spin modules
\begin{align*}
\mathbb S^+=\wedge^{\rm even}W,\quad \mathbb S^-=\wedge^{\rm odd}W,
\end{align*} which are the irreducible representations of $\mathfrak{so}(2m,\mathbb C)$ \cite[p. 118, Theorem 1]{Delanghe}.
For $n=2m+1,$  $W={\rm span}_{\mathbb C}\left\{f_j,j=1,\cdots,m,\right\},$ $W'={\rm span}_{\mathbb C}$ $\left\{\bar f_{j},j=1,\cdots,m,\right\},$ $U=\left\{f_{2m+1}\right\},$   are maximally isotropic subspaces of   $\mathbb C_n$ and $\mathbb C^n=W\oplus W'\oplus U,$ where $U$ is a $1$-dimensional space perpendicular to them and \begin{align*}
f_{m+1}=\mathbf i e_{2m+1}.
\end{align*}   We have one spin module \begin{align*}\mathbb S=\wedge^\bullet W,\end{align*} which is the irreducible representation of $\mathfrak{so}(2m+1,\mathbb C).$
%i.e. $\bar f_j$ is obtained from $f_j$ by taking the bar-map in $\mathbb C_{2n}.$
$f_j$'s satisfy relations \begin{equation}\begin{aligned}\label{fjfk}
f_jf_k+f_kf_j=0,\ \bar f_j\bar f_k+\bar f_k\bar f_j=0,\ \bar f_j  f_k+  f_k\bar f_j=0,\ &{\rm for}\ j\neq k,\\f_{m+1}f_k+f_kf_{m+1}=0,\  f_{m+1}\bar f_k+\bar f_k\bar f_{m+1}=0,\ &{\rm for\ any}\ k.
\end{aligned}\end{equation}

%with highest weights $\alpha=\frac12\left(L_1+\cdots+L_m\right)$and $\beta=\frac12\left(L_1+\cdots+L_{m-1}-L_n\right)$ (\cite[Propostion 20.15]{FH}).
 %with highest weight $\alpha=\frac12\left(L_1+\cdots+L_m\right)$ (\cite[Lemma 20.7]{FH}).

The action of Clifford algebra   on the spin module $\mathbb S^\pm$ is given by $\gamma\left(f_s\right)\in {\rm End}\ \mathbb S^\pm.$
For any $a\in \mathbb C_n,$ $\gamma(a):(\wedge W)I\rightarrow(\wedge W)I$ is given by
\begin{align*}
\gamma(a)(bI)=abI,
\end{align*}
where $I=I_1I_2\cdots I_m,$ with $I_j=\bar f_jf_j,j=1,\cdots,m,$ is a primitive idempotent in $\mathbb C_n$ (cf. \cite[p. 115]{Delanghe}). A primitive idempotent of  $\mathbb C_n$ is a nonzero idempotent $I$ such that $I\mathbb C_n$ is indecomposable as a right $\mathbb C_n$-module; that is, such that $I\mathbb C_n$ is not a direct sum of two nonzero submodules. Equivalently, $I$ is a primitive idempotent if it cannot be written as $I = J+K,$ where $J$ and $K$ are nonzero orthogonal idempotents in $\mathbb C_n.$ Then  $\mathbb S^+,\mathbb S^-$ and $\mathbb S$ can be realized as $\wedge^{\rm even}WI,\wedge^{\rm odd}WI$ and $\wedge^\bullet WI,$ respectively.
For any $\alpha=\left\{\alpha_1,\cdots,\alpha_h\right\}\subset\{1,\cdots,m\}$ with $\alpha_1<\alpha_2<\cdots<\alpha_h,$ set
\begin{align}\label{ba}
f_\alpha:=f_{\alpha_1}\wedge\cdots\wedge f_{\alpha_h},
\end{align}
which constitute a basis of $\mathbb C_n.$
%To simplify the notation, we use the symbols $\mathbb S^\pm$ to denote the two spinor modules for $n$ odd, and the same symbols are used for $n$ even with the convention that $\mathbb S^+$ and $\mathbb S^-$ are isomorphic. If $V_\lambda$ is an irreducible ${\rm GL}(2)$ module with the highest $\lambda,$ the   irreducible $G_0$-module is denoted shortly by $V_\lambda^\pm\cong V_\lambda\otimes\mathbb S^\pm,$ or just by the Young diagram corresponding to the partition $\lambda.$

\begin{thm}\label{cli}{\rm(cf. \cite[p.116, Theorem 1]{Delanghe})}
For $j=1,\cdots,m,$\\
{\rm (1)} $\gamma\left(f_j\right)f_\alpha I =\left(f_j\wedge f_\alpha\right)I;$\\
{\rm (2)} $\gamma\left(\bar f_j\right)f_\alpha I=\bar f_j\left(f_\alpha\right)I;$\\
{\rm (3)} $\gamma\left(f_{m+1}\right)f_\alpha I=(-1)^{\sharp \alpha}f_\alpha I,$\\
where \begin{align*} \bar f_j\left(f_\alpha \right)=\left\{
\begin{array}{lcl}&0,&{\rm if}\ j\notin \alpha,\\&(-1)^{t+1}f_{\alpha'},&{\rm if}\ j=\alpha_t, \alpha'=\alpha\setminus\{\alpha_t\}\end{array}\right.,\end{align*} and   $\sharp\alpha=h$ if $f_\alpha=f_{\alpha_1}\wedge\cdots\wedge f_{\alpha_h}.$
\end{thm}
Denote $\gamma_j:=\gamma\left(e_{j}\right).$ Since
\begin{align*}
e_{2j-1}=f_j-\bar f_j,\quad  e_{2j}=-\mathbf i\left(f_j-\bar f_j\right),\quad  e_{2m+1}=-\mathbf if_{m+1},
\end{align*}
%Set \begin{equation*}\begin{aligned} \gamma_{2j-1}=&\gamma\left(e_{2j-1}\right)=\gamma\left(f_{j}\right) -\gamma\left(\bar f_{j}\right),\quad \gamma_{2j}=\gamma\left(e_{2j}\right)=-\mathbf i\left(\gamma\left(f_{j}\right)+\gamma\left(\bar f_{j}\right)\right),\\\gamma_{2m+1}=&\gamma\left(e_{2m+1}\right)=-\mathbf i\gamma\left(f_{m+1}\right). \end{aligned}\end{equation*}
it follows form (\ref{fjfk}) that $\gamma_j:\mathbb S^\pm\rightarrow\mathbb S^\mp,$
and
\begin{align}\label{comm}
\gamma_j\gamma_k+\gamma_k\gamma_j=-2\delta_{jk}\bf 1.
\end{align}
Define an Hermitian  inner product on $\mathbb S^\pm$ by  \begin{align*}
\left\langle f_\alpha,f_\beta\right\rangle_{\mathbb S^\pm}:=\delta_{\alpha\beta}.
\end{align*}
\begin{lem} The adjoint $\gamma_k^*$   of $\gamma_k$ with respect to this inner product  is given by
\begin{align}\label{star}\gamma_k^*=-  \gamma_k.\end{align}
\end{lem}
\begin{proof}
We only need to prove
\begin{equation*}
\begin{aligned}
\left\langle\gamma_kf_\alpha,f_\beta\right\rangle_{\mathbb S^\pm}=-\left\langle f_\alpha, \gamma_k f_\beta\right\rangle_{\mathbb S^\mp},
\end{aligned}\end{equation*}
where $f_\alpha$ is the basis of  $\mathbb C_n$ given in (\ref{ba}). It is easy to see that
$\left\langle\gamma_{2j-1}f_\alpha,f_\beta\right\rangle_{\mathbb S^\mp}$ $\neq0$   only if $\alpha\cup\{j\}=\beta$ or $\beta\cup\{j\}=\alpha.$ If $\alpha\cup\{j\}=\beta$ and $j$ is the $k$-th element of $\beta,$ we have $$\left\langle\gamma_{2j-1}f_\alpha,f_\beta\right\rangle_{\mathbb S^\pm}=(-1)^{k-1}=-\left\langle f_\alpha,\gamma_{2j-1}f_\beta\right\rangle_{\mathbb S^\mp},$$ by Theorem \ref{cli}. So $\gamma_{2j-1}^*=-\gamma_{2j-1},$ for $j=1,\cdots,m.$ Similarly, we can prove the $\beta\cup\{j\}=\alpha$ case.

Also $\left\langle\gamma_{2j}f_\alpha,f_\beta\right\rangle_{\mathbb S^\mp}\neq0$   only if $\alpha\cup\{j\}=\beta$ or $\beta\cup\{j\}=\alpha.$ If $\alpha\cup\{j\}=\beta$ and $j$ is the $k$-th element of $\beta,$ we have $$\left\langle\gamma_{2j}f_\alpha,f_\beta\right\rangle_{\mathbb S^\pm}=-\mathbf i(-1)^{k-1}=-\left\langle f_\alpha,\gamma_{2j}f_\beta\right\rangle_{\mathbb S^\mp},$$ by Theorem \ref{cli}. So $\gamma_{2j}^*=-\gamma_{2j},$ for $j=1,\cdots,m.$ Similarly, we can prove the $\beta\cup\{j\}=\alpha$ case.

$\left\langle\gamma_{2m+1}f_\alpha,f_\beta\right\rangle_{\mathbb S^\mp}\neq0$  only if $\alpha=\beta.$ In this case
$$\left\langle\gamma_{2m+1}f_\alpha,f_\alpha\right\rangle_{\mathbb S^\pm}=-\mathbf i(-1)^{\sharp\alpha}=-\left\langle f_\alpha,\gamma_{2m+1}f_\alpha\right\rangle_{\mathbb S^\mp},$$ by Theorem \ref{cli}. So $\gamma_{2m+1}^*=-\gamma_{2m+1}.$
%For $v\in \mathbb S^\pm$ and $v'\in \mathbb S^\mp,$ we have \begin{equation*} \begin{aligned} \left\langle\gamma_kv,v'\right\rangle_{\mathbb S^\pm}=-\left\langle v, \gamma_k v'\right\rangle_{\mathbb S^\mp}=:\left\langle v,\gamma_k^*v'\right\rangle_{\mathbb S^\mp} \end{aligned}\end{equation*} by Theorem \ref{cli}.
The lemma is proved.
\end{proof}
%Recall that a partition of an integer $N > 0$ in $k$ terms is a sequence of $k$ positive integers $\lambda_1\geq\cdots\geq \lambda_k>0$  such that $N = \sum_i\lambda_i.$ Such $k$-tuple can be represented by a Young diagram, by collecting a set of $N$ boxes into $k$ rows in which each row contains $\lambda_i$ elements. We will use the symbol $\lambda$ to denote equivalently the partition and the list of positive integers or the Young diagram.  %We will call $N = |\lambda|$ the degree of $\lambda,$ while max$\left\{i|\lambda_i > 0\right\}$ will be its depth. For example, the partition $\lambda= (2, 1)$ has degree $N = 3$, depth $2.$ The correspond representation is denoted by $V_\lambda.$ Young diagrams with at most $2$ rows parametrize irreducible representations of ${\rm SL}(2).$ %The module $\mathbb V_\lambda$ corresponding to the Young diagram $\lambda$ has the highest weight $\lambda.$ The operator $\mathscr D_j$ maps functions on $\mathbb R^{2n}$  with values in $ V_\lambda\otimes\mathbb S^\pm$ to functions on $\mathbb R^{2n}$ with values in $V_{\lambda+1}\otimes\mathbb S^\mp$ where
 %and $\mathbb S^\pm$ is the representation of ${\rm Spin}(n).$
\subsection{Operators in the Dirac complex of two vector variables}

Note that
\begin{equation}\begin{aligned}\label{Delta}
\nabla_A\nabla_A=&\sum_{i,j=1}^n\gamma_i{\partial_{A i}}\gamma_j \partial_{A j}=-\sum_{j=1}^n{\partial^2_{A j}}:=\Delta_A,
\end{aligned}\end{equation}
by (\ref{comm}), for $A=0,1.$

Define the Hermitian inner product on $L^2\left(\mathbb R^{2n},\mathscr V_j\right),$ for $j=1,3,$ by
\begin{equation*}
(\phi,\psi)_{j}:=\sum_{A=0,1}\int_{\mathbb R^{2n}}\left\langle\phi_{A}(\mathbf x),{\psi_A(\mathbf x)}\right\rangle_{\mathbb S^-}\hbox{d}V(\mathbf x),
\end{equation*}
for  $\phi,\psi\in L^2\left(\mathbb R^{2n},\mathscr V_j\right),$ and define the Hermitian inner product on $L^2\left(\mathbb R^{2n},\mathscr V_j\right),$ for $j=0,2,$ by  \begin{equation*}
(\Phi,\Psi)_{j}=\int_{\mathbb R^{2n}}\left\langle\Phi(\mathbf x),{\Psi(\mathbf x)}\right\rangle_{\mathbb S^+}\hbox{d}V(\mathbf x),
\end{equation*}
for  $\Phi,\Psi\in L^2\left(\mathbb R^{2n},\mathscr V_j\right),$ where \begin{align*}
\hbox{d}V(\mathbf x)=\hbox{d}x_{01}\wedge\cdots\wedge\hbox{d}x_{0n} \wedge\hbox{d}x_{11}\wedge\cdots\wedge\hbox{d}x_{1n},
\end{align*}
is the standard volume form on $\mathbb R^{2n}.$ We can rewrite the operator in (\ref{D}) into the matrix form as follows. $\mathscr D_0:\Gamma\left(\mathbb S^+\right)\rightarrow\Gamma\left(\mathbb C^2\otimes \mathbb S^-\right)$ is given by
\begin{equation}
\begin{aligned}\label{D2}
\left(\begin{matrix}  \nabla_0\\\nabla_1\end{matrix}\right)= \left(\begin{matrix}  \sum_j \gamma_j {\partial_{0j}}\\ \sum_k\gamma_k{\partial_{1k}}\end{matrix}\right), \end{aligned}\end{equation}
$\mathscr D_1:\Gamma\left(\mathbb C^2\otimes\mathbb S^-\right)\rightarrow\Gamma\left(\mathbb C^2\otimes\mathbb S^+\right)$ is given by
\begin{equation*}
\begin{aligned} \left(\begin{array}{cc} -\nabla_1\nabla_0 & \nabla_0\nabla_0\\ -\nabla_1\nabla_1 & \nabla_0\nabla_1\end{array}\right)=\left(\begin{array}{cc} -\sum_{j,k}\gamma_k\gamma_j{\partial_{0j}}{\partial_{1k}} & \Delta_0\\ -\Delta_1 &\sum_{j,k}\gamma_j\gamma_k{\partial_{0j}}{\partial_{1k}}\end{array}\right),
\end{aligned}\end{equation*}
while
$\mathscr D_2:\Gamma\left(\mathbb C^2\otimes\mathbb S^+\right)\rightarrow\Gamma\left(\mathbb S^-\right)$ is given by
\begin{equation}
\begin{aligned}\label{d22}\left(-\nabla_1,\nabla_0\right) =\left(-\sum_k\gamma_k{\partial_{1k}},\sum_j\gamma_j {\partial_{0j}}\right).
\end{aligned}\end{equation}
Let $\mathscr D_j^*$ be the formal adjoint of the operator $\mathscr D_j$ in (\ref{D}), i.e.
\begin{align}\label{Dstar}
\left(\mathscr D_j\phi,\psi\right)_{j+1}=\left(\phi,\mathscr D_j^*\psi\right)_j
\end{align}
for any $\phi\in C_0^\infty \left(\mathbb R^{2n},{\mathscr V_j}\right),\psi\in C_0^\infty\left(\mathbb R^{2n},{\mathscr V_{j+1}}\right).$
\begin{prop}\label{p23}
For any $\phi\in C_0^\infty\left(\mathbb R^{2n},\mathscr V_0\right),$ we have
\begin{align*}
\mathscr D_0^*\mathscr D_0\phi=\Delta\phi,
\end{align*}
where $\Delta=\Delta_0+\Delta_1$ is the Laplacian operator on $\mathbb R^{2n}.$
\end{prop}
\begin{proof}
For $\phi\in C_0^\infty\left(\mathbb R^{2n},\mathbb S^\pm\right),$ and $\psi\in C_0^\infty\left(\mathbb R^{2n},\mathbb S^\mp\right),$ we have
\begin{equation*}
\begin{aligned}
\int_{\mathbb R^{2n}}\left\langle\nabla_A\phi,{\psi_A}\right\rangle_{\mathbb S^\mp}\hbox{d}V(\mathbf x)=&\int_{\mathbb R^{2n}}\sum_i\left\langle\gamma_i\partial_{Ai}\phi, {\psi_A}\right\rangle_{\mathbb S^\mp}\hbox{d}V(\mathbf x)\\=&\int_{\mathbb R^{2n}}\sum_i\left\langle\phi, \gamma_i\partial_{Ai}{\psi_A}\right\rangle_{\mathbb S^\pm}\hbox{d}V(\mathbf x)\\=&\int_{\mathbb R^{2n}}\left\langle\phi,{\nabla_A\psi_A}\right\rangle_{\mathbb S^\pm}\hbox{d}V(\mathbf x)
\end{aligned}\end{equation*}
by using (\ref{star}) and   Stokes's formula, i.e. the formal adjoint $\nabla_A^*$ of $\nabla_A$ satisfies
\begin{align}\label{nablastar}
\nabla_A^*=\nabla_A.
\end{align}Thus, we have \begin{equation*}
\begin{aligned}
\left(\mathscr D_0\phi,\psi\right)=\sum_{A=0,1}\int_{\mathbb R^{2n}}\left\langle\nabla_A\phi,{\psi_A}\right\rangle_{\mathbb S^-}\hbox{d}V(\mathbf x)=\left(\phi,\mathscr D_0^*\psi\right),
\end{aligned}\end{equation*} i.e.  \begin{align}
\mathscr D_0^*\psi=\sum_{A=0,1}\nabla_A\psi_A.
\end{align}
Thus \begin{equation}\begin{aligned}\label{dsd}
\mathscr D_0^*\mathscr D_0\phi=\sum_{A=0,1}\nabla_A\nabla_A\phi=\sum_{A=0,1}\sum_{j,k=0}^{n-1} \gamma_j{\partial_{Aj}}\gamma_k{\partial_{Ak}}\phi=\Delta\phi.
\end{aligned}\end{equation}
The proposition is proved.
\end{proof}
 Recall that the   symbol of a matrix differential operator $$\mathscr D = \sum_{|\alpha|\leq n}B_{\alpha_{A_1k_1}\cdots\alpha_{A_Nk_N}}(\mathbf x) \partial^{\alpha_{A_1k_1}}_{A_1k_1}\cdots \partial^{\alpha_{A_Nk_N}}_{A_Nk_N}: \Gamma(\Omega,V)\rightarrow \Gamma(\Omega,V')$$ at $(\mathbf x, \nu)$ for  $\mathbf x\in\Omega\subset\mathbb R^{2n},\mathbf 0\neq\nu\in\mathbb R^{2n}$  is defined as
{\small\begin{align}\label{symbol}
\sigma(\mathscr D)(\mathbf x;\nu):=\sum_{|\alpha|= n}B_{\alpha_{A_1k_1}\cdots\alpha_{A_Nk_N}}(\mathbf x) \left(\frac{\nu_{A_1k_1}}{\mathbf i}\right)^{\alpha_{A_1k_1}}\cdots \left(\frac{\nu_{A_Nk_N}}{\mathbf i}\right)^{\alpha_{A_Nk_N}}:V\rightarrow V',
\end{align}}
where $A_l\in\{0,1\},k_l\in\{1,\cdots,n\},$ and  $B_{\alpha_{A_1k_1}\cdots\alpha_{A_Nk_N}}$ is a linear transformation from vector space $V$ to $V'.$
A differential complex
\begin{equation*}
0\rightarrow C^\infty\left(\Omega,V_0\right)\xrightarrow{\mathscr{D}_{0}}\cdots \xrightarrow{\mathscr{D}_{n-1}} C^\infty\left(\Omega,V_n\right)\rightarrow 0,
\end{equation*} is called \emph{elliptic} if its symbol sequence
\begin{equation*}
0\rightarrow V_0 \xrightarrow{\sigma\left(\mathscr{D}_{0}\right){(\mathbf x;\nu)}}\cdots \xrightarrow{\sigma\left(\mathscr{D}_{n-1}\right){(\mathbf x;\nu)}} V_n\rightarrow 0,
\end{equation*} is exact for any $\mathbf x\in\Omega,\nu\in\mathbb R^{2n}\setminus\{\mathbf 0\},$ i.e. $$\ker\sigma\left(\mathscr{D}_{l}\right){(\mathbf x;\nu)}={\rm Im}\ \sigma\left(\mathscr{D}_{l-1}\right){(\mathbf x;\nu)}.$$
The  symbol of differential operator $\nabla_A$ is
 \begin{align}\label{simp}
\sigma\left(\nabla_A\right)(\mathbf x;\nu)=-{\mathbf i}\sum_{j=1}^n {\gamma_j\nu_{A j}}:=\nu_A
\end{align}

\begin{thm}\label{pell}
The complex {\rm(\ref{co})} with operators given by {\rm (\ref{D})} is an elliptic differential complex.
\end{thm}
\begin{proof}
Firstly, we prove  \begin{align}\label{d10}
\mathscr{D}_{l+1}\circ\mathscr{D}_l=0
\end{align}  for each $l=0,1.$

When   $l=0,$ noting  that $\mathscr D_0f=\left(\begin{matrix}  \nabla_0f\\\nabla_1f\end{matrix}\right)$ for $f\in\Gamma\left(\mathscr V_0\right),$ we have
\begin{equation*}\begin{aligned}
\left(\mathscr D_1\circ\mathscr D_0f\right)_0=&\left(\nabla_0\nabla_{0}\nabla_{1} -\nabla_{1}\nabla_{0}\nabla_{0}\right)f =\left(\Delta_{0}\nabla_{1}-\nabla_1\Delta_{0} \right)f=0,\\ \left(\mathscr D_1\circ\mathscr D_0f\right)_1=&\left(\nabla_0\nabla_{1}\nabla_{1} -\nabla_{1}\nabla_{1}\nabla_{0}\right)f =\left(\nabla_{0}\Delta_{1}-\Delta_{1}\nabla_0 \right)f=0,
\end{aligned}\end{equation*} by the definition of $\mathscr D_1$ in (\ref{D}) and using  (\ref{Delta}) and
\begin{align}\label{AB}
\nabla_B\Delta_A=\Delta_A\nabla_B,\quad A,B=0,1,
\end{align}
since $\Delta_A$ is a scalar differential  operator of constant coefficients. So $\mathscr{D}_{1}\circ\mathscr{D}_0=0.$

 When  $l=1,$ for $g=\left(\begin{matrix}  g_0\\g_1\end{matrix}\right)\in\Gamma\left(\mathscr V_1\right),$ we have
\begin{equation*}\begin{aligned}
\mathscr D_2\circ \mathscr D_1g=&\nabla_{0}\left(\mathscr D_1g\right)_1-\nabla_{1}\left(\mathscr D_1g\right)_0\\=&\nabla_0\left(\nabla_0\nabla_{1} g_{1}-\nabla_{1}\nabla_{1}g_{0}\right)-\nabla_1\left(\nabla_0\nabla_{0} g_{1}-\nabla_{1}\nabla_{0}g_{0}\right)\\=&\Delta_0\nabla_{1} g_{1}-\nabla_0\Delta_1g_{0}-\nabla_1\Delta_{0} g_{1}+\Delta_1\nabla_0 g_{0}=0,
\end{aligned}\end{equation*}
by the definition of $\mathscr D_2$ in (\ref{D}) and (\ref{AB}). So $\mathscr{D}_{2}\circ\mathscr{D}_1=0.$

Secondly,  let us  prove that {\rm(\ref{co})} is    elliptic, i.e.
 the symbol sequence
\begin{equation}\begin{aligned}\label{simco}
0\rightarrow\mathbb S^+\xrightarrow{\sigma_{0}(\nu)} \mathbb C^2\otimes\mathbb S^+\xrightarrow{\sigma_{1}(\nu)} \mathbb C^2\otimes\mathbb S^- \xrightarrow{\sigma_{2}(\nu)} \mathbb S^+\rightarrow0,
\end{aligned}\end{equation}is exact,  $\nu\in\mathbb R^{2n}\setminus\{\mathbf 0\},$ where $\sigma_l(\nu)=\sigma\left(\mathscr D_l\right)(\mathbf x;\nu)$ is independent of   $\mathbf x\in\mathbb R^{2n}.$

It follows from  $\mathscr D_{l+1}\circ\mathscr D_l=0$ that $\sigma_{l+1}\circ\sigma_l=0,$ i.e. ${\rm Im}\ \sigma_l(\nu)\in\ker\sigma_{l+1}(\nu).$ So  we only need to prove that $\sigma_0$ is injective, $\ker\sigma_l\subset{\rm Im}\ \sigma_{l-1},$ for $l=1,2,$ and $\sigma_2$ is surjective.\\
(1) Denote  $\left|\nu_A\right|^2:=\nu_{A1}^2+\cdots+\nu_{An}^2.$ Then
\begin{align}\label{nunu}
\nu_A\nu_A=-{\mathbf i}\sum_k {\gamma_k\nu_{A k}}\left(-{\mathbf i}\sum_j{\gamma_j\nu_{A j}}\right)=\left|\nu_A\right|^2{\rm id}_{\mathbb S^\pm}.
\end{align}  By  definition, for $\eta\in\mathscr V_0$   $$\sigma_0(\nu)\eta=\left(\begin{matrix}\nu_0\eta\\ \nu_1\eta\end{matrix}\right) \in\mathbb C^2\otimes\mathbb S^+.$$ For any $\eta\in\ker \sigma_0(\nu),$ we have $\nu_A\eta=0,$ for $A=0,1.$ Then \begin{align}\label{vv}
0=\nu_A\nu_A\eta=|\nu_A|^2\eta,
\end{align}
by (\ref{nunu}). So $\eta=0.$ Hence $\sigma_0$ is injective.
\vskip 5mm\noindent(2) By definition, for $\xi=\left(\begin{matrix}\xi_0 \\ \xi_1 \end{matrix}\right)\in\mathscr V_1,$ $$\sigma_1(\nu)\xi=\left(\begin{matrix}\nu_{0}\nu_{0}\xi_1- \nu_{1}\nu_{0}\xi_0\\ \nu_{0}\nu_{1}\xi_1- \nu_{1}\nu_{1}\xi_0\end{matrix}\right) \in\mathbb C^2\otimes\mathbb S^-.$$ Note that \begin{equation}\begin{aligned}\label{sig1}
\left(\sigma_1(\nu)\xi\right)_0=&|\nu_0|^2\xi_1- \nu_{1}\nu_{0}\xi_0,\\ \left(\sigma_1(\nu)\xi\right)_1=&\nu_{0}\nu_{1}\xi_1-|\nu_1|^2\xi_0,
\end{aligned}\end{equation} by (\ref{nunu}). If   $\xi=\left(\begin{matrix}\xi_0 \\ \xi_1 \end{matrix}\right)\in \ker \sigma_1(\nu)$ with $\left|\nu_0\right|\neq0,$ let
\begin{align*}
\eta=\frac{\nu_{0}\xi_0}{|\nu_0|^2}\in\mathscr V_0.
\end{align*}
We claim that $\sigma_0\eta=\xi.$
Since $\xi=\left(\begin{matrix}\xi_0 \\ \xi_1 \end{matrix}\right)\in \ker \sigma_1(\nu),$ we have
$|\nu_0|^2\xi_1-\nu_{1}\nu_{0}\xi_0=0$ by (\ref{sig1}). Then
\begin{align*}
\nu_{1}|\nu_0|^2\xi_1=\nu_{1}\nu_{1}\nu_{0}\xi_0=\nu_{0}|\nu_1|^2\xi_0,
\end{align*}
by (\ref{nunu}), i.e.
\begin{align}\label{eta}
\frac{\nu_{0}\xi_0}{|\nu_0|^2}=\frac{\nu_{1}\xi_1}{|\nu_1|^2}=\eta,
\end{align}if $\left|\nu_1\right|\neq0.$
Then
\begin{equation}\begin{aligned}\label{xixi}
\left(\sigma_0(\nu)\eta\right)_0 =\frac{\nu_{0}\nu_{0}}{|\nu_0|^2}\xi_0=\xi_0,\\ \left(\sigma_0(\nu)\eta\right)_1 =\frac{\nu_{1}\nu_{1}}{|\nu_0|^2}\xi_1=\xi_1,
\end{aligned}\end{equation}
by (\ref{nunu}) and (\ref{eta}). If $\left|\nu_1\right|=0,$ we must have $\xi_1=0,$ and so $\left(\sigma_0(\nu)\eta\right)_1=\nu_1\eta=0=\xi_1,$ i.e. (\ref{xixi})  also holds. Thus $\sigma_0\eta=\xi$ and so $\ker\sigma_1\subset{\rm Im}\ \sigma_0.$

 If $\left|\nu_0\right|=0,$ we get $\left|\nu_1\right|^2\xi_0=0$ by (\ref{sig1}). Since $\nu\neq\mathbf 0,$ we have $\xi_0=0.$ Then we have $$\sigma_0(\nu)\left(\nu_1^{-1}\xi_1\right)=\left(\begin{matrix}0 \\ \xi_1 \end{matrix}\right),$$ since $\nu_1$ is reversible.
\vskip 5mm\noindent
(3) By  definition, for $\zeta=\left(\begin{matrix}\zeta_0\\\zeta_1\end{matrix}\right)\in\mathscr V_1,$ we have   $$\sigma_2(\nu)\zeta=\nu_{0}\zeta_1-\nu_{1}\zeta_0\in \mathbb C\otimes\mathbb S^+.$$  For any   $\zeta=\left(\begin{matrix}\zeta_0\\\zeta_1\end{matrix}\right)\in \ker \sigma_2(\nu),$ let $\xi=\left(\begin{matrix}\xi_0\\\xi_1\end{matrix}\right)\in \mathscr V_1$ given by
\begin{equation}\begin{aligned}\label{xi}
\xi_0=-\frac{\zeta_1}{|\nu_0|^2 +|\nu_1|^2},\quad\xi_1=\frac{\zeta_0}{|\nu_0|^2+|\nu_1|^2}.
\end{aligned}\end{equation}
Since $\zeta=\left(\begin{matrix}\zeta_0\\\zeta_1\end{matrix}\right)\in \ker \sigma_2(\nu),$ we have
\begin{align}\label{sig2}
\sigma_2(\nu)\zeta=\nu_{0}\zeta_1-\nu_{1}\zeta_0=0.
\end{align}
Then  we have
\begin{equation*}\begin{aligned}
\left(\sigma_1(\nu)\xi\right)_0=&\nu_{0}\nu_{0}\xi_1- \nu_{1}\nu_{0}\xi_0=\frac{|\nu_0|^2\zeta_0}{|\nu_0|^2+|\nu_1|^2}+ \frac{ \nu_{1}\nu_{0}\zeta_1}{|\nu_0|^2+|\nu_1|^2}\\=& \frac{|\nu_0|^2\zeta_0}{|\nu_0|^2+|\nu_1|^2}+\frac{ \nu_{1}\nu_{1}\zeta_0}{|\nu_0|^2+|\nu_1|^2}=\zeta_0,
\end{aligned}\end{equation*}
  by using (\ref{xi})-(\ref{sig2}). Similarly,
\begin{equation*}\begin{aligned}
\left(\sigma_1(\nu)\xi\right)_1=&\nu_{0}\nu_{1}\xi_1- \nu_{1}\nu_{1}\xi_0=\frac{ \nu_{0}\nu_{1}\zeta_0}{|\nu_0|^2 +|\nu_1|^2}+\frac{|\nu_1|^2\zeta_1}{|\nu_0|^2+|\nu_1|^2}\\=&\frac{ \nu_{0}\nu_{0}\zeta_1}{|\nu_0|^2+|\nu_1|^2}+ \frac{|\nu_1|^2\zeta_1}{|\nu_1|^2+|\nu_1|^2}=\zeta_1.
\end{aligned}\end{equation*}
Thus $\sigma_1(\nu)\xi=\zeta$ and so $\ker\sigma_2(\nu)\subset{\rm Im}\sigma_1.$
\vskip 5mm\noindent   (4) For any   $\rho\in \mathscr V_3,$ let $\zeta=\left(\begin{matrix}\zeta_0\\\zeta_1\end{matrix}\right)\in \mathscr V_2$ be  given by
\begin{equation*}\begin{aligned}
\zeta_0=-\frac{\nu_{1}\rho}{|\nu_0|^2+|\nu_1|^2}, \quad\zeta_1=\frac{ \nu_{0}\rho}{|\nu_0|^2+|\nu_1|^2}.
\end{aligned}\end{equation*}
Then we have
\begin{equation*}\begin{aligned}
\sigma_2(\nu)\zeta=&\nu_{0}\zeta_1-\nu_{1}\zeta_0 =\frac{\nu_{0}\nu_{0}\rho}{|\nu_0|^2+|\nu_1|^2}+ \frac{\nu_{1}\nu_{1}\rho}{|\nu_0|^2+|\nu_1|^2}\\=& \frac{|\nu_0|^2\rho}{|\nu_0|^2+|\nu_1|^2} +\frac{|\nu_1|^2\rho}{|\nu_0|^2+|\nu_1|^2}=\rho.
\end{aligned}\end{equation*}
So $\sigma_2$ is surjective.
\end{proof}

\section{Solutions to the non-homogeneous Dirac  equations and the Hartogs'  phenomenon}
\subsection{Solutions to the non-homogeneous   equations}
The natural Hodge Laplacian associated to the differential  complex (\ref{co}) should be
\begin{align*}
\widetilde\Box_j=\mathscr D_{j-1}\mathscr D_{j-1}^*+\mathscr D_j^*\mathscr D_j,
\end{align*}
for $j=1,2.$ But $\mathscr D_1$ is a differential operator of second order while $\mathscr D_0,\mathscr D_2$ are of first order, the principal symbols of $\widetilde\Box_1$ and $\widetilde\Box_2$ are degenerate and so $\widetilde\Box_1$ and $\widetilde\Box_2$ are not uniformly elliptic.
So it is better to  consider the   Laplacians (\ref{hodge}) of forth order associated to the differential  complex (\ref{co}).

\begin{prop}\label{pbox}
\begin{equation}\begin{aligned}\label{box2}
\Box_1=\left(\mathscr D_0\mathscr D_0^*\right)^2+\mathscr D_1^*\mathscr D_1=\left(\begin{array}{cc} \Delta^2&\\  & \Delta^2\end{array}\right), \\\Box_2=\mathscr D_1\mathscr D_1^*+\left(\mathscr D_2^*\mathscr D_2\right)^2=\left(\begin{array}{cc} \Delta^2&\\  & \Delta^2\end{array}\right).
\end{aligned}\end{equation}
\end{prop}
\begin{proof}

%For $\phi\in C_0^\infty\left(\mathbb R^{2\times n},\mathbb S^\pm\right),$ and $\psi\in C_0^\infty\left(\mathbb R^{2\times n},\mathbb S^\mp\right),$ we have
%\begin{equation}\begin{aligned}
%\left(\gamma_k\phi,\psi\right)_{\mathbb S^\pm}=\int_{\mathbb R^{2\times n}}\gamma_k\phi\cdot\overline{\psi}\hbox{d}v(  x)=-\int_{\mathbb R^{2\times n}}\phi\cdot\overline{\overline{\gamma_k}\psi}\hbox{d}v(  x)=\left(\phi,\gamma_k^*\psi\right)_{\mathbb S^\mp}
%\end{aligned}\end{equation}
%which follows from Stokes's formula. Then we have $\gamma_k^*=-\gamma_k.$
By using  (\ref{nablastar}), it is direct to check by definition (\ref{Dstar}) that the  formal adjoints of $\mathscr D_0,\mathscr D_1,\mathscr D_2$ are given by
\begin{equation}
\begin{aligned}\label{ddd}
\mathscr D_0^*=& \left(\nabla_0, \nabla_1 \right),\\ \mathscr D_1^*=&\sum_{j,k}\left(\begin{array}{cc} -\nabla_0\nabla_1 & -\Delta_1\\ \Delta_0 & \nabla_1\nabla_0\end{array}\right),\\ \mathscr D_2^*=& \left(\begin{matrix}-\nabla_1\\  \nabla_0\end{matrix}\right),
\end{aligned}\end{equation}
respectively.
Recall  $\mathscr D_0^*\mathscr D_0=\Delta$ by (\ref{dsd}).
Compositions of matrix-valued differential operators give us  that \begin{equation}
\begin{aligned}\label{DD*}
\mathscr D_0\mathscr D_0^*=&\sum_{j,k} \left(\begin{array}{cc} \Delta_0&\nabla_0\nabla_1\\ \nabla_1\nabla_0&\Delta_1\end{array}\right),\\ \mathscr D_1^*\mathscr D_1=&\sum_{j,k} \left(\begin{array}{cc} \Delta_0\Delta_1+\Delta_1^2& -\nabla_0\nabla_1\left(\Delta_0+\Delta_1\right)\\ -\nabla_1\nabla_0\left(\Delta_0+\Delta_1\right)&\Delta_0 \Delta_1+\Delta_0^2\end{array}\right), \\ \mathscr D_1\mathscr D_1^*=&\sum_{j,k} \left(\begin{array}{cc} \Delta_0\Delta_1+\Delta_0^2&\nabla_1\nabla_0 \left(\Delta_0+\Delta_1\right)\\ \nabla_0\nabla_1 \left(\Delta_0+\Delta_1\right)&\Delta_0\Delta_1+\Delta_1^2\end{array} \right), \\\mathscr D_2^*\mathscr D_2=&  \left(\begin{array}{cc} \Delta_1& -\nabla_1\nabla_0\\ -\nabla_0\nabla_1 &\Delta_0\end{array}\right),
\end{aligned}\end{equation}
by expressions in (\ref{ddd}) and (\ref{D2})-(\ref{d22}), and then we have
\begin{equation*}\begin{aligned}
\left(\mathscr D_0^*\mathscr D_0\right)^2=&\Delta^2,\\ \left(\mathscr D_0\mathscr D_0^*\right)^2=&\sum_{j,k} \left(\begin{array}{cc} \Delta_0^2+\Delta_0\Delta_1&\nabla_0\nabla_1\left(\Delta_0+\Delta_1\right) \\ \nabla_1\nabla_0 \left(\Delta_0+\Delta_1\right)&\Delta_0\Delta_1 +\Delta_1^2\end{array}\right), \\\left(\mathscr D_2^*\mathscr D_2\right)^2=&\sum_{j,k} \left(\begin{array}{cc} \Delta_1^2+\Delta_0\Delta_1&- \nabla_1\nabla_0\left(\Delta_0+\Delta_1\right)\\ -\nabla_0\nabla_1 \left(\Delta_0+\Delta_1\right)&\Delta_0\Delta_1+\Delta_0^2\end{array} \right).
\end{aligned}\end{equation*}
Thus (\ref{box2}) follows.
\end{proof}

So $\Box_j$ are uniformly elliptic differential operator of fourth order.
It is known that
\begin{align}\label{G0}
G_0=-\frac{C_{2n}}{|\mathbf x|^{2n-4}},
\end{align}
 is the fundamental solution of the  operator $\Delta^2$ on $\mathbb R^{2n}$ for some positive constant $C_{2n}.$ %and$|x|=\left(x_{00}^2+\cdots x_{0(n-1)}^2+x_{10}^2+\cdots+ x_{1(n-1)}^2\right)^{\frac12}.$
Let
\begin{align}\label{G1}
G_1=G_2=\left(\begin{array}{cc} G_0&\\  & G_0\end{array}\right).
\end{align}

The following proposition for homogeneous distributions and singular integral operators is well known. %
\begin{prop}{\rm(\cite[Proposition 2.4.7]{Grafakos})}\label{PV} Let $K\in C^\infty\left(\mathbb R^{2n}\setminus\{\mathbf 0\}\right)$   be a homogeneous function of degree $k-2n.$ Let $\mathbf K$ be the operator defined by $\mathbf K\phi = \phi *K.$
Then, for $\phi\in C_0^\infty\left(\mathbb R^{2n}\right)$,  $A_1,\cdots,A_k\in\{0,1\},j_1,\cdots, j_k\in\{1,\cdots, 2n\},$
\begin{align}\label{3.2}
\partial_{A_1j_1}\cdots\partial_{A_kj_k}(\mathbf K\phi)=P.V.\left(\phi*\partial_{A_1j_1}\cdots\partial_{A_kj_k}K\right) +a_{A_1j_1\cdots A_kj_k}\phi,
\end{align}
and each term in {\rm (\ref{3.2})} is $C^\infty$ and the identity holds as $C^\infty$ functions, where $a_{A_1j_1\cdots A_kj_k}$ is a constant.

Moreover, $\partial_{A_1j_1}\cdots\partial_{A_kj_k}  K$  is a Calderon--Zygmund kernel on $\mathbb R^{2n}.$ The singular integral operator $f\rightarrow P.V.$ $ \left(f *\partial_{A_1j_1}\cdots\partial_{A_kj_k}K\right)$ is bounded
on $L^p$ for $1 < p <\infty.$ {\rm (\ref{3.2})} holds as $L^p$ functions.
\end{prop}
\begin{prop}\label{GG}
The Laplacian $\Box_j$ has the inverse $\mathbf G_j$ in $L^2\left(\mathbb R^{2n},\mathscr V_j\right), j=0,1,2,$ which is a convolution operator with kernel $G_j.$ The kernel $G_j$ is a smooth  homogeneous functions of degree $4-2n.$  Moreover, $\mathbf G_j$ can be extended to a bounded linear operator from $L^p\left(\mathbb R^{2n},\mathscr V_j\right)$ to $W^{4,p}\left(\mathbb R^{2n},\mathscr V_j\right),$ for $1 < p <\infty.$ If a
$\mathscr V_j$-valued function $f$ is in $C_0^k\left(\mathbb R^{2n},\mathscr V_j\right),$ for any non-negative integer $k,$ we have $\mathbf G_jf\in C^{k+3}\left(\mathbb R^{2n},\mathscr V_j\right).$
\end{prop}
\begin{proof}
For   $f\in C_0^\infty\left(\mathbb R^{2n},\mathscr V_j\right).$ Let \begin{align}
\mathbf G_jf=\int_{\mathbb R^{2n}}G_j(\mathbf x-\mathbf y)f(\mathbf y)\hbox{d}V(\mathbf y)%\in L^2\left(\left(\mathbb R^n\right)^2,\mathscr V_j\right)
,\ j=0,1,2.
\end{align}\label{gG}
Note that for $u\in C_0^\infty\left(\mathbb R^{2n}\right)$
\begin{align*}
\int_{\mathbb R^{2n}}G_0(\mathbf x-\mathbf y)\Delta^2u(\mathbf y){\rm d}V(\mathbf y)=u(\mathbf x),
\end{align*}
since $G_0$ is the fundamental solution to $\Delta^2.$
Then $\mathbf G_j$ is the inverse operator of  Laplacian $\Box_j,$  i.e.
\begin{align*}
\mathbf G_j\Box_jf=\Box_j\mathbf G_jf=f.
\end{align*}
$G_0$ satisfies the following decay estimates
\begin{align}\label{partial}
\left|\partial_{{A_1j_1}} \cdots\partial_{{A_mj_m}} G_0\right|\leq \frac{C_{A_1j_1\cdots A_mj_m}}{|\mathbf x|^{2n-4+m}},
\end{align}
for some constant $C_{A_1j_1\cdots A_mj_m}>0$ depending on  $A_1,\cdots,A_m\in\{0,1\},$ $j_1,\cdots,$ $j_m\in\{1,\cdots,n\}.$
By   Proposition \ref{PV}, the convolution with a homogeneous function of degree $-2n+k$ can be extended to a bounded linear operator from $L^p\left(\mathbb R^{2n}\right)$ to $W^{k,p}\left(\mathbb R^{2n}\right).$

When $f\in C_0^{k}\left(\mathbb R^{2n},\mathscr V_j\right),$ it follows from (\ref{3.2}) that  $\mathbf G_jf\in C^{k+3}\left(\mathbb R^{2n},\mathscr V_j\right)$  by differentiation. The proposition is proved.
\end{proof}

\begin{prop}\label{p34}
A  monogenic  function $f\in\mathcal O(\Omega)$ on a domain $\Omega\in\mathbb R^{2n}$ is real analytic.
\end{prop}
\begin{proof}
Note that
\begin{align}\label{box0}
\Box_0f=\left(\mathscr D_0^*\mathscr D_0\right)^2f=\Delta^2f=0,
\end{align}
in the sense of distributions. Thus $f$ is biharmonic and so it is  real analytic at each point $\mathbf x\in\Omega.$
\end{proof}

Now we can prove the Theorem \ref{t31}.\vskip 5mm
{\it Proof of Theorem \ref{t31}.} We prove the case $j=1,$ the case  $j=2$ is similar. Recall that  $\mathbf G_1$ is the inverse operator of $\Box_1$ on $C_0^\infty\left(\mathbb R^{2n},\mathscr V_1\right).$  Set
\begin{align*}
u:=\mathscr D_0^*\mathscr D_0\mathscr D_0^*\mathbf G_1f\in C^\infty\left(\mathbb R^{2n},\mathscr V_0\right).
\end{align*}
Note that
\begin{equation}\begin{aligned}\label{d1g1}
\mathscr D_1\Box_1=&\mathscr D_1\left(\mathscr D_1^*\mathscr D_1+\left(\mathscr D_0\mathscr D_0^*\right)^2\right)=\mathscr D_1\mathscr D_1^*\mathscr D_1\\=&\left(\mathscr D_1\mathscr D_1^*+\left(\mathscr D_2^*\mathscr D_2\right)^2\right)\mathscr D_1=\Box_2\mathscr D_1,
\end{aligned}\end{equation}
on $C_0^\infty\left(\mathbb R^{2n},\mathscr V_2\right)$  by $\mathscr D_1\mathscr D_0=0,\mathscr D_2\mathscr D_1=0.$ Then we have
\begin{equation*}\begin{aligned}
\Box_2 \left(\mathbf G_2\mathscr D_1f-\mathscr D_1\mathbf G_1f\right)=\mathscr D_1f-\mathscr D_1\Box_1\mathbf G_1f=0,
\end{aligned}\end{equation*}
by using (\ref{d1g1}). By Proposition \ref{pbox}, each entry of the $2\times 2$ matrix $\mathbf G_2\mathscr D_1f-\mathscr D_1\mathbf G_1f$ is biharmonic. On the other hand, we have $$\left\|\mathbf G_2\mathscr D_1f(\mathbf x)-\mathscr D_1\mathbf G_1f(\mathbf x)\right\|\leq\frac{C}{\left(1+|\mathbf x|\right)^{2n-2}},$$ for some constant $C>0$ by (\ref{partial}).  Thus   $\mathbf G_2\mathscr D_1f-\mathscr D_1\mathbf G_1f$ is constant by Liouville-type theorem   \cite{Huilgol},
i.e.
\begin{align}\label{GD}
\mathbf G_2\mathscr D_1=\mathscr D_1\mathbf G_1,
\end{align}
on $C_0^\infty\left(\mathbb R^{2n},\mathscr V_2\right).$ Thus \begin{align}\label{DDDDG}
\mathscr D_0u=\mathscr D_0\mathscr D_0^*\mathscr D_0\mathscr D_0^*\mathbf G_1f=\left(\left(\mathscr D_0\mathscr D_0^*\right)^2+\mathscr D_1^*\mathscr D_1\right)\mathbf G_1f=f,
\end{align}
by \begin{align}\label{DGf}\mathscr D_1\mathbf G_1f=\mathbf G_2\mathscr D_1f=0.\end{align}
Thus $u=\mathscr D_0^*\mathscr D_0\mathscr D_0^*\mathbf G_1f$ satisfies (\ref{duf}) for $f\in C_0^\infty\left(\mathbb R^{2n},\mathscr V_1\right).$

Note that $\mathbf G_j, j = 0, 1, 2,$ is formally self-adjoint in the following sense:
\begin{align*}
\left(\mathbf G_j\phi,\psi\right)_{\mathscr V_j}=\left(\phi,\mathbf G_j\psi\right)_{\mathscr V_j},
\end{align*}
for any $\phi,\psi\in C_0^\infty\left(\mathbb R^{2n},\mathscr V_j\right),$ by the explicit expression of $G_j,j=0,1,2,$ (\ref{G0})-(\ref{G1}). Also $\mathbf G_j$ is bounded from $L^2\left(\mathbb R^{2n},\mathscr V_j\right)$ to $W^{4,2}\left(\mathbb R^{2n},\mathscr V_j\right)$ by Proposition \ref{GG}. So $\mathbf G_j$ can be extended to a bounded operator from $W^{-4,2}\left(\mathbb R^{2n},\mathscr V_j\right)$ to $L^2\left(\mathbb R^{2n},\mathscr V_j\right)$ by the duality argument. In particular, it is bounded from $W^{-2,2}\left(\mathbb R^{2n},\mathscr V_j\right)$ to $L^2\left(\mathbb R^{2n},\mathscr V_j\right).$ Therefore, the identity (\ref{GD}) holds as   boun-ded linear operators on $L^2\left(\mathbb R^{2n},\mathscr V_1\right)$ and so the identity (\ref{DGf}) holds for $f\in L^2\left(\mathbb R^{2n},\mathscr V_1\right).$ Thus (\ref{DDDDG}) holds as $L^2$ functions for  $f\in L^2\left(\mathbb R^{2n},\mathscr V_1\right)$ satisfying $\mathscr D_1f = 0$ in the sense of distributions. Thus    $u =\mathscr D_0^*\mathscr D_0\mathscr D_0^*\mathbf G_1f$ satisfies the  equation (\ref{duf}).

Suppose that $f$ is supported in $\Omega\Subset\mathbb R^{2n}.$  Since the integral kernel $K(\mathbf x)$ of $\mathscr D_0^*\mathscr D_0\mathscr D_0^*\mathbf G_1$ decays as $|\mathbf x|^{-2n+1}$ for large $|\mathbf x|$ by decay estimate (\ref{partial}) and $f$ is compactly supported, we see that
\begin{align}\label{k}
\left|u(\mathbf x)\right|=\left|\int_{\mathbb R^{2n}}K\left(\mathbf x-{\mathbf y}\right)f\left({\mathbf y}\right)\hbox{d}V\left({\mathbf y}\right)\right|\leq\frac{C}{\left(1+|\mathbf x|\right)^{2n-1}},
\end{align}
for some constant $C>0.$
%by Theorem \ref{t41},
So  $\lim_{|\mathbf x|\rightarrow\infty}u(\mathbf x)=0.$ For $\mathbf x=\left(\mathbf x_0,\mathbf x_1\right),$  if $|\mathbf x_0|$ is so large that $$\left(\left\{\mathbf x_0\right\}\times\mathbb R^n\right)\cap\bar\Omega=\emptyset,$$ then $u\left(\mathbf x_0,\mathbf x_{1}\right)$ is a monogenic function in $\mathbf x_{1},$ which vanishes at infinity. Since
\begin{align}\label{har}
\Delta_1u\left(\mathbf x_0,\mathbf x_1\right)=0,
\end{align}by (\ref{Delta})
each component of $u\left(\mathbf x_0, \cdot\right)$ is a biharmonic function on $\mathbb R^n$ vanishing at infinity and so is bounded. Hence $u\left(\mathbf x_0, \cdot\right)\equiv0$ for $|\mathbf x_0|$ large by Liouville-type theorem again. By (\ref{box0}), we  see that $u(\cdot)$ is monogenic on $\mathbb R^{2n}\setminus\Omega.$ So $u\equiv0$ on the unbounded connected component of $\mathbb R^{2n}\setminus\Omega$ by the identity theorem for real analytic function, since  $u$ is real analytic on $\mathbb R^{2n}\setminus\Omega$ by Proposition \ref{p34}. The continuity of $u$ follows from the formula (\ref{k}). The theorem is proved.
\qed

\subsection{The Bochner--Martinelli formula   on $\mathbb R^{2n}$}
Denote
\begin{align}\label{H()}
H(\mathbf x):=\mathscr D_0^*\mathscr D_0\mathscr D_0^*G_1(\mathbf x),
\end{align}
whose entries are $C^\infty\left(\mathbb R^{2n}\setminus\{\mathbf 0\}\right)$ homogeneous functions of degree $1-2n,$ where $G_1(\cdot)$ is given by (\ref{G1}).
By (\ref{DD*}), we have \begin{equation*}\begin{aligned}
H(\mathbf x)=&\mathscr D_0^*\mathscr D_0\mathscr D_0^*\left(\begin{array}{cc} G_0&\\  & G_0\end{array}\right)\\=&\left(\nabla_0,\nabla_1\right) \left(\begin{array}{cc} \Delta_0&\nabla_0\nabla_1\\ \nabla_1\nabla_0&\Delta_1\end{array}\right) \left(\begin{array}{cc} G_0&\\  & G_0\end{array}\right) \\=&\left(\nabla_0\Delta,\nabla_1\Delta\right)\left(\begin{array}{cc} G_0&\\  & G_0\end{array}\right) \\=&\left(\sum_j\gamma_j\partial_{0j}\Delta G_0,\sum_k\gamma_k\partial_{1k}\Delta G_0\right)\\=&4(n-2)C_{2n}\left(\sum_j\gamma_j\partial_{0j}|\mathbf x|^{2-2n},\sum_k\gamma_k\partial_{1k}|\mathbf x|^{2-2n}\right) \\=&-8(n-1)(n-2)C_{2n}\left(\sum_j\frac{\gamma_jx_{0j}}{|\mathbf x|^{2n}},\sum_k\frac{\gamma_kx_{1k}}{|\mathbf x|^{2n}}\right),
\end{aligned}\end{equation*}
for $\mathbf x\in\mathbb R^{2n}\setminus\{\mathbf 0\},$ where $C_{2n}$ is the positive constant   in (\ref{G0}).
 \begin{thm}\label{t41}
Suppose that $\Omega$ is a bounded domain in $\mathbb R^{2n}$ with $C^2$  boundary, and suppose that $f:\Omega\rightarrow \mathscr V_1$  is a continuous function of class $W^{1,2}(U)$ for a domain $U$ with $\overline \Omega\subset U\subset\mathbb R^{2n}.$  Let $n(\mathbf x)$ be the unit outer normal vector to the surface $\partial\Omega,$   we have
\begin{align*}
f(\mathbf x)=-\int_{\partial\Omega}H(\mathbf x-\mathbf y)\left(\begin{matrix} n_0\\ n_1\end{matrix}\right)f(\mathbf y){\rm d}S(\mathbf y)+\int_{\Omega}H(\mathbf x-\mathbf y)\mathscr D_0f(\mathbf y){\rm d}V(\mathbf y),
\end{align*}
where ${\rm d}S$ is the surface measure and  $%P_0(n)=\mathbf i \left(\begin{matrix} \sigma\left(\nabla_0\right)(n;\mathbf x)\\ \sigma\left(\nabla_1\right)(n;\mathbf x)\end{matrix}\right)=
\left(\begin{matrix} n_0\\ n_1\end{matrix}\right)=\left(\begin{matrix} \sum_j\gamma_jn_{0j}\\\sum_j\gamma_jn_{1j}\end{matrix}\right).$ %is defined by {\rm (\ref{simp})} with $n$ by $\nu$.
\end{thm}
\begin{proof}
Since
\begin{align*}
\mathscr D_0\Box_0=\mathscr D_0\mathscr D_0^*\mathscr D_0\mathscr D_0^*\mathscr D_0=\left(\left(\mathscr D_0\mathscr D_0^*\right)^2 +\mathscr D_1^*\mathscr D_1\right)\mathscr D_0=\Box_1\mathscr D_0,
\end{align*}
by $\mathscr D_1\mathscr D_0=0.$ As in the proof of Theorem \ref{t31},  we have $\mathbf G_1\mathscr D_0=\mathscr D_0\mathbf G_0,$ which holds as operator from  $L^{2}\left(\mathbb R^{2n},\mathscr V_0\right)$ to $W^{3,2}\left(\mathbb R^{2n},\mathscr V_0\right)$ by Proposition \ref{GG}. So
\begin{align}\label{id}
{\rm id}=\left(\mathscr D_0^*\mathscr D_0\right)^2\mathbf G_0=\mathscr D_0^*\mathscr D_0\mathscr D_0^* \mathbf G_1\mathscr D_0
\end{align}
on $L^2.$ It is sufficient to show the theorem   for $f\in C^\infty(U).$ Suppose that $\epsilon$ is sufficiently small, let $\phi_\epsilon(\mathbf x)=\epsilon^{-2n}\phi\left(\epsilon^{-1}\mathbf x\right).$ Then $\left(\chi_\Omega*\phi_\epsilon\right)\cdot f\in C_0^\infty\left(\mathbb R^{2n},\mathscr V_1\right).$ Apply (\ref{id}) to $\left(\chi_\Omega*\phi_\epsilon\right)\cdot f$ to get
\begin{equation}\begin{aligned}\label{chi}
\left(\chi_\Omega*\phi_\epsilon\cdot f\right)(\mathbf x)=&\mathscr D_0^*\mathscr D_0\mathscr D_0^* \mathbf G_1\mathscr D_0\left(\chi_\Omega*\phi_\epsilon\cdot f\right)(\mathbf x)\\=&\mathscr D_0^*\mathscr D_0\mathscr D_0^*\int_{\mathbb R^{2n}}G_1(\mathbf x-\mathbf y)\chi_\Omega*\phi_\epsilon(\mathbf y)\mathscr D_0f(\mathbf y){\rm d}V (\mathbf y)\\&+\mathscr D_0^*\mathscr D_0\mathscr D_0^*\int_{\mathbb R^{2n}}G_1(\mathbf x-\mathbf y)\chi_\Omega*\left(\mathscr D_0\phi_\epsilon(\mathbf y)\right)f(\mathbf y){\rm d}V(\mathbf y).
\end{aligned}\end{equation}
It follows from the definition of $H(\cdot)$ in (\ref{H()})  that its entries are $C^\infty\left(\mathbb R^{2n}\right.$ $\left.\setminus \{\mathbf 0\}\right)$ homogeneous functions of degree $1 - 2n,$ since $G_1(\cdot)$ is a   matrix with each entry a $C^\infty\left(\mathbb R^{2n}\setminus \{\mathbf 0\}\right)$ homogeneous function of degree $4 - 2n$ by Proposition \ref{GG}.
Note that for a fixed $\mathbf x\in\Omega,$ $\chi_\Omega*\phi_\epsilon(\mathbf x)=1$   if $\epsilon$ is sufficiently small.    (\ref{chi}) can be rewritten as
\begin{equation}\begin{aligned}\label{H}
f(\mathbf x)=&\int_{\mathbb R^{2n}}H(\mathbf x-\mathbf y)\chi_\Omega*\phi_\epsilon(\mathbf y)\mathscr D_0f(\mathbf y){\rm d}V(\mathbf y)\\&+\int_{\mathbb R^{2n}}H(\mathbf x-\mathbf y)\chi_\Omega*\mathscr D_0\phi_\epsilon(\mathbf y)f(\mathbf y){\rm d}V(\mathbf y).
\end{aligned}\end{equation}
The first term in (\ref{H}) converges obviously for smooth $f$ as $\epsilon\rightarrow0.$ For the second term,  to see $\mathscr D_0\phi_\epsilon *\chi_\Omega$ converging to a matrix-valued measure on $\partial \Omega,$ note that
\begin{equation}\begin{aligned}\label{int}
\frac{\partial}{\partial x_{Aj}} \left(\phi_{\epsilon}*\chi_{\Omega}\right)(\mathbf x)=&\int_{\Omega} \frac{\partial\phi_{\epsilon}}{\partial x_{Aj}}(\mathbf x-\mathbf y){\rm d}V(\mathbf y)\\=&-\int_{\Omega}\frac{\partial}{\partial y_{Aj}}\left(\phi_\epsilon(\mathbf x-\mathbf y)\right){\rm d}V(\mathbf y)\\=&-\int_{\partial\Omega}\phi_\epsilon(\mathbf x-\mathbf y)n_{Aj}(\mathbf y){\rm d}S(\mathbf y),
\end{aligned}\end{equation}
where $n_{Aj}(\mathbf y)$  is the $(Aj)$-th component of the unit outer normal vector to $\partial \Omega, A=0,1,j=1,\cdots,n.$ By (\ref{int}), we see that the support  of $\mathscr D_0\phi_\epsilon*\chi_\Omega$ is a small neighborhood of $\partial \Omega$ if $\epsilon$ is sufficiently small. So for a function $g$ continuous on a neighborhood of $\partial \Omega,$ we have
\begin{equation*}
\begin{aligned}
&\int_{\mathbb R^{2n}}g(\mathbf x)\mathscr D_0\phi_\epsilon*\chi_\Omega(\mathbf x){\rm d}V(\mathbf x)\\=&-\int_{\mathbb R^{2n}}g(\mathbf x){\rm d}V(\mathbf x)\int_{\partial \Omega}\phi_{\epsilon}(\mathbf x-\mathbf y)\left(\begin{matrix} n_0\\ n_1\end{matrix}\right)(\mathbf y){\rm d}S(\mathbf y)\\=&-\int_{\partial \Omega}\left(\int_{\mathbb R^{2n}}g(\mathbf x)\phi_{\epsilon}(\mathbf x-\mathbf y){\rm d}V(\mathbf x)\right)\left(\begin{matrix} n_0\\ n_1\end{matrix}\right)(\mathbf y){\rm d}S(\mathbf y)\\\rightarrow&-\int_{\partial \Omega}g(\mathbf y)\left(\begin{matrix} n_0\\ n_1\end{matrix}\right)(\mathbf y){\rm d}S(\mathbf y),
\end{aligned}\end{equation*}
as $\epsilon\rightarrow0,$ by $\int_{\mathbb R^{2n}}g(\mathbf x)\phi_\epsilon(\mathbf x-\mathbf y){\rm d}V(\mathbf x)=g*\phi_\epsilon(\mathbf y)\rightarrow g(\mathbf y).$ The result follows.\end{proof}

\begin{rem} By Theorem \ref{t41}, for a    monogenic function  $f \in C\left(\overline{\Omega},\mathbb{S}^+\right)\cap\ C^1\left({\Omega},\mathbb{S}^+\right),$ where  $\Omega$ is an open bounded set of  $\mathbb R^{2n}$ with $C^1$ boundary,  the
Cauchy-type formula holds, i.e.
\begin{align*}
f(\mathbf x)=-\int_{\partial\Omega}H(\mathbf x-\mathbf y)\left(\begin{matrix} n_0\\ n_1\end{matrix}\right)f(\mathbf y){\rm d}S(\mathbf y).
\end{align*}
\end{rem}

%The following theorem gives the solution to the non-homogeneous equation (\ref{duf})  under the compatibility condition.

\subsection{Hartogs' extension phenomenon   for    monogenic functions}
Now we can prove  the    Hartogs'  extension phenomenon   for  monogenic functions.
\vskip 5mm{\it Proof of Theorem \ref{hartogs}.}
By Proposition \ref{p23}, a monogenic function is biharmonic and so it is $C^\infty.$ Let $\chi\in C_0^\infty(\Omega)$ be equal to $1$ in a neighborhood of $K.$ Set $\tilde u:=(1-\chi)u,$ which vanishes on $K.$ Then $\tilde u\in C^\infty(\Omega,\mathbb S^+).$ Let $\tilde U$ be the solution to the non-homogeneous
\begin{align*}
\mathscr D_0\tilde U=\mathscr D_0\tilde u=-\mathscr D_0\chi u=f,
\end{align*}
where $f,$ defined as $0$ in $K$ and outside $\Omega,$ has components in $C_0^\infty\left(\mathbb R^{2n},\mathbb S^+\right)$ and satisfies the compatible condition $\mathscr D_1f=\mathscr D_1\mathscr D_0\tilde u=0.$ It follows from Theorem \ref{t31} that   there exists such a solution $\tilde U$ which vanishes in the unbounded component of the complement of the support of $\chi.$ Then the function
\begin{align*}
U=\tilde u-\tilde U
\end{align*}
is  monogenic in $\Omega$ since $\mathscr D_0\left(\tilde u-\tilde U\right)=0$ on $\Omega.$ Note that $\tilde U$ on $\tilde \Omega=\mathbb R^{2n}\setminus {\rm supp}\chi$ and supp$\chi\subset K'$ for some compact set $K\subset K'\Subset \Omega.$ So $U=u$ on $\Omega\setminus K'.$ Then $U=u$ on $\Omega\setminus K$ by the identity theorem. The theorem is proved.
\qed

\section{The generalization of Malgrange's vanishing theorem and the Hartogs--Bochner extension for monogenic functions}
A \emph{cohomological complex} of topological vector spaces is a pair $(E^\bullet, {\rm d}),$ where $E^\bullet=\left(E^q\right)_{q\in\mathbb Z}$ is a sequence of topological vector spaces and ${\rm d} = \left(\hbox{d}^q\right)_{q\in\mathbb Z}$ is a sequence of continuous linear
maps $\hbox{d}^q:E^q\rightarrow E^{q+1}$ satisfying $\hbox{d}^{q+1}\circ\hbox{d}^q=0.$ Its cohomology groups $H^q\left(E^\bullet\right)$ are the quotient spaces $\ker \hbox{d}^q/{\rm Im}\ \hbox{d}^{q-1},$ endowed with the quotient topology.
A \emph{homological complex} of topological vector spaces is a pair $\left(E_\bullet, \hbox{d}\right)$ where $E_\bullet=\left(E_q\right)_{q\in\mathbb Z}$ is a sequence of topological vector spaces and ${\rm d} = \left(\hbox{d}_q\right)_{q\in\mathbb Z}$ is a sequence of continuous linear maps $\hbox{d}_q:E_q\rightarrow E_{q-1}$ satisfying $\hbox{d}_{q-1}\circ\hbox{d}_q=0.$ Its homology groups $H_q\left(E_\bullet\right)$ are the
quotient spaces $\ker \hbox{d}_{q-1}/{\rm Im}\ \hbox{d}_q,$ endowed with the quotient topology. The dual complex of a cohomological complex $(E^\bullet, {\rm d})$ of topological vector spaces is the homological complex $(E'_\bullet, {\rm d}'),$  where $E'_\bullet=\left(E'_q\right)_{q\in\mathbb Z}$ with $E'_q$  the  dual of $E_q$ and $\hbox{d}'=\left(\hbox{d}_q'\right)_{q\in\mathbb Z}$ with $\hbox{d}_q'$ the transpose map of $\hbox{d}_q.$

Recall that a \emph{Fr\'echet--Schwartz space} is a topological vector space whose topology is defined by an increasing sequence of seminorms such that the unit ball with respect to the
seminorm is relatively compact for the topology associated to the previous seminorm. %A Fr\'echet--Schwartz space and the dual of a Fr\'echet--Schwartz space are both reflexive.
We need the following abstract duality theorem.

\begin{thm}{\rm(\cite[Theorem 1.6]{Laurent})}\label{dual}
Let $(E^\bullet, {\rm d})$ be a cohomological complex of Fr\'echet--Schwartz spaces or of dual of Fr\'echet--Schwartz spaces and let  $(E_\bullet, {\rm d})$  be its dual complex.
For each $q\in\mathbb Z$, the following assertions are equivalent:\\
{\rm (1)} ${\rm Im}\ {\rm d}^q = \left\{g\in E^{q+1}|\langle g, f\rangle=0{\rm\ for\ any}\ f\in \ker {\rm d}_q' \right\};$\\
{\rm (2)} $H^{q+1}\left(E^\bullet\right)$ is separated;\\
{\rm (3)} ${\rm d}^q$ is a topological homomorphism;\\
{\rm (4)} ${\rm d}'_q$ is a topological homomorphism;\\
{\rm (5)} $H_q\left(E^\bullet\right)$ is separated;\\
{\rm (6)} ${\rm Im}\ {\rm d}'_q = \left\{f\in E'_{q}|\langle f, g\rangle=0 {\rm\ for\ any}\ g\in \ker {\rm d}^q \right\}.$
\end{thm}
A continuous linear map $\Psi$ between topological vector spaces $L_1$ and $L_2$ is called a
{\it topological homomorphism} if for each open subset $U\subset L_1,$ the image $\Psi(U)$ is an open subset of $\Psi\left(L_1\right)$. It is known that if $L_1$ is a Fr\'echet space, $\Psi$ is a topological homomorphism if and only if $\Psi\left(L_1\right)$ is closed \cite{Schaefe}. See e.g. \cite{Brinkschulte,Laurent,wang29} for applications of abstract duality theorem to $\bar \partial$ or $\bar \partial_b$-complex or tangential $k$-Cauchy--Fueter complex, respectively. We adapt their methods to the  Dirac complex.

For a complex vector space $V,$ let $\mathcal E\left(\mathbb R^{2n}, V\right)$ be the space of smooth $V$-valued functions with the topology of uniform convergence on compact sets of the functions and all their derivatives. Endowed with this topology $\mathcal E\left(\mathbb R^{2n}, V\right)$ is a Fr\'echet--Schwartz space. Let $\mathcal D\left(\mathbb R^{2n}, V\right)$ be the space of compactly supported elements of $\mathcal E\left(\mathbb R^{2n}, V\right).$ For a compact subset
$K$ of $\mathbb R^{2n},$ let $\mathcal D_K\left(\mathbb R^{2n}, V\right)$ be the closed subspace of $\mathcal E\left(\mathbb R^{2n}, V\right)$ with support in $K$ endowed with the induced topology. Choose $\left\{K_n\right\}_{n\in\mathbb N}$ an exhausting sequence of compact subsets of $\mathbb R^{2n}.$ Then $\mathcal D\left(\mathbb R^{2n}, V\right)=\cup_{n=1}^\infty\mathcal D_{K_n}\left(\mathbb R^{2n}, V\right)$ \cite{Laurent}. We put on $\mathcal D\left(\mathbb R^{2n}, V\right)$ the strict inductive limit topology defined by the Fr\'echet--Schwartz spaces $\mathcal D_{K_n}\left(\mathbb R^{2n}, V\right).$ Denote by $\mathcal E'\left(\mathbb R^{2n}, V\right)$ the dual of $\mathcal E\left(\mathbb R^{2n}, V\right)$ and $\mathcal D'\left(\mathbb R^{2n}, V\right)$ the dual of $\mathcal D\left(\mathbb R^{2n}, V\right).$

%Denote $V_0=\mathbb V_{0}^+,V_1=\mathbb V_{1}^-,V_2=\mathbb V_{21}^-,V_3=\mathbb V_{22}^+.$
Noting that  the dual of $V$ is $V$ itself as a finite dimensional complex vector space. The dual of the    complex $$0\rightarrow\mathcal D\left(\mathbb R^{2n},{\mathscr V}_0\right)\rightarrow\mathcal D\left(\mathbb R^{2n},{\mathscr V}_1\right)\rightarrow\mathcal D\left(\mathbb R^{2n},{\mathscr V}_2\right)\rightarrow\mathcal D\left(\mathbb R^{2n},{\mathscr V}_3\right)\rightarrow0,$$ is
{\small\begin{equation}\begin{aligned}\label{dco}
0\leftarrow\mathcal D'\left(\mathbb R^{2n},{\mathscr V}_0\right)\xleftarrow{\widehat{\mathscr{D}}_{0}}\mathcal D'\left(\mathbb R^{2n},{\mathscr V}_1\right)\xleftarrow{\widehat{\mathscr{D}}_{1}} \mathcal D'\left(\mathbb R^{2n},{\mathscr V}_2\right) \xleftarrow{\widehat{\mathscr{D}}_{2}} \mathcal D'\left(\mathbb R^{2n},{\mathscr V}_3\right)\leftarrow 0.
\end{aligned}\end{equation}}
%where $\widehat{\mathscr V}_0=\mathbb C\otimes\mathbb S^-,$ $\widehat{\mathscr V}_1=\mathbb C^2\otimes\mathbb S^-,$ $\widehat{\mathscr V}_2=\mathbb C^2\otimes\mathbb S^+,$ $\widehat{\mathscr V}_3=\mathbb C\otimes\mathbb S^-.$???
%Here $\mathcal E'\left(\mathbb R^{2n}, \mathscr V_j\right)$ and  $\mathcal D'\left(\mathbb R^{2n}, \mathscr V_j\right)$   can be identified with   $\mathcal E'\left(\mathbb R^{2n}, \widehat{\mathscr V}_j\right)$ and  $\mathcal D'\left(\mathbb R^{2n}, \widehat{\mathscr V}_j\right),$ respectively.
%\begin{prop}???
%The dual of Dirac complex is
%{\small\begin{equation}\begin{aligned}
%0\rightarrow\Gamma\left(\Omega,\mathbb C\otimes\mathbb S^-\right)\xrightarrow{\mathscr{D}_{0}} \Gamma\left(\Omega,\mathbb C^2\otimes\mathbb S^+\right)\xrightarrow{\mathscr{D}_{1}}\Gamma\left(\Omega,\mathbb C^2\otimes\mathbb S^-\right) \xrightarrow{\mathscr{D}_{2}} \Gamma\left(\Omega,\mathbb C\otimes\mathbb S^+\right)\rightarrow 0.\end{aligned}\end{equation}}\end{prop}
%Recall that $\mathscr D_0:\mathcal E\left(\left(\mathbb R^n\right)^2,\mathscr V_0\right)\rightarrow \mathcal E\left(\left(\mathbb R^n\right)^2,\mathscr V_1\right)$ is given by \begin{align*} \mathscr D_0F=\left(\nabla_0F,\nabla_1F\right),\end{align*}for $F\in\mathcal E\left(\left(\mathbb R^n\right)^2,\mathscr V_0\right).$
The dual can be realization as follows. For $F\in\mathcal E\left(\mathbb R^{2n}, {\mathscr V}_j\right),$ we can  define  a functional on $\mathcal D\left(\mathbb R^{2n},\mathscr V_j\right)$ by
\begin{align*}
\langle F,\phi\rangle:=\int_{\mathbb R^{2n}}\langle F,\phi\rangle_{\mathscr V_j}\hbox{d}V,
\end{align*}
for $\phi\in\mathcal D\left(\mathbb R^{2n},\mathscr V_j\right).$ %by identify $\mathbb S^+$ with $\mathbb S^-.$ Similarly, for $\mathbb F=\left(\mathbb F_1,\mathbb F_2\right)\in\mathcal E\left(\mathbb R^{2n},\widehat{\mathscr V}_1\right),$ we can  define  a functional on $\mathcal D\left(\mathbb R^{2n},\mathscr V_1\right)$ by \begin{align*} (\mathbb F,\phi):=\int_{\mathbb R^{2n}}\langle \mathbb F,\psi\rangle\hbox{d}V, \end{align*} for $\psi\in\mathcal E\left(\mathbb R^{2n},\mathscr V_1\right).$
Then for $\mathbb F=\left(\begin{matrix}\mathbb F_1\\\mathbb F_2\end{matrix}\right)\in\mathcal E\left(\mathbb R^{2n}, {\mathscr V}_{1}\right)$ and $f\in\mathcal D\left(\mathbb R^{2n},\mathscr V_0\right),$ we have
\begin{equation*}
\begin{aligned}
\left\langle\widehat{\mathscr D}_0\mathbb F,f\right\rangle=&\left\langle\mathbb F,\mathscr D_0f\right\rangle=\int_{\mathbb R^{2n}}\left\langle\mathbb F,\mathscr D_0f\right\rangle_{\mathscr V_1}\hbox{d}V\\=&\sum_{A=0,1}\int_{\mathbb S^{+}} \left\langle \mathbb F_A,\nabla_Af\right\rangle_{\mathbb S^+}\hbox{d}V=\int_{\mathbb R^{2n}}\left\langle\sum_{A=0,1}\nabla_A\mathbb F_A,{f}\right\rangle_{\mathscr V_0}\hbox{d}V.
\end{aligned}\end{equation*}
So when acting on smooth elements $\widehat{\mathscr D}_0$ in (\ref{dco}) is a differential operator:
\begin{align}\label{hat}
\widehat{\mathscr D}_0\mathbb F=\sum_{A=0,1}\nabla_A\mathbb F_A.
\end{align}

\begin{prop}
Suppose  $K\Subset\mathbb R^{2n}.$  Then
\begin{align*}
\|\mathscr D_0f\|_{L^2}^2= \|f\|_{W^{1,2}}^2-\|f\|_{L^2}^2,
\end{align*}
for $f\in C_0^\infty\left(K,\mathscr V_0\right).$
\end{prop}
\begin{proof}
By (\ref{dsd}),
\begin{equation*}\begin{aligned}
\|\mathscr D_0f\|_{L^2}^2=&\int_{\mathbb R^{2n}}\left\langle\mathscr D_0 f,\mathscr D_0 f\right\rangle_{\mathscr V_1}{\rm d}V =\int_{\mathbb R^{2n}}\left\langle\mathscr D_0^*\mathscr D_0 f,f\right\rangle_{\mathbb S^+}{\rm d}V\\=&\int_{\mathbb R^{2n}}\left\langle\Delta f,f\right\rangle_{\mathbb S^+}{\rm d}V=\sum_{A,j}\left\|\frac{\partial f}{\partial x_{Aj}}\right\|_{L^2}^2= \|f\|_{W^{1,2}}^2-\|f\|_{L^2}^2,
\end{aligned}\end{equation*}
by the definition of $W^{1,2}$-norm. The proposition is proved.
\end{proof}
Then  we have the following estimate by the standard procedure  (cf. \cite{Nacinovich}).  %\cite[Corollary 7.1]{wang29} for right-type group case).
\begin{cor}
For any $K\Subset\mathbb R^{2n},$ there are constants $C_{s,K}>0,c_{s,K}\geq0$  such that
\begin{align}\label{esi}
C_{s,K}\|f\|_{W^{s,2}}^2+\left\|\mathscr D_0f\right\|_{W^{s,2}}^2\geq c_{s,K}\|f\|_{W^{s+1,2}}^2,
\end{align}
for any $f\in\mathcal E'\left(\mathbb R^{2n}, \mathscr V_0\right)$ with $\mathscr D_0f\in W^{s,2}\left(\mathbb R^{2n}, \mathscr V_0\right)$   and supp$f\subset K.$
\end{cor}

\begin{thm}\label{th73}
$\mathscr D_0 :\mathcal E'\left(\mathbb R^{2n},\mathscr V_0\right)\rightarrow \mathcal E'\left(\mathbb R^{2n},\mathscr V_1\right)$ and $\mathscr D_0 :\mathcal D\left(\mathbb R^{2n},\mathscr V_0\right)\rightarrow \mathcal D\left(\mathbb R^{2n}, \mathscr V_1\right)$ have closed ranges.
\end{thm}
\begin{proof}
Let $\{f_\nu\}$ be a sequence in $\mathcal E'\left(\mathbb R^{2n},\mathscr V_0\right)$ such that $\mathscr D_0f_\nu$ convergence in $\mathcal E'\left(\mathbb R^{2n},\mathscr V_0\right),$ i.e. all $\mathscr D_0f_\nu$ are supported in a fixed compact subset $K\Subset\mathbb R^{2n}$ and there is a $s\in\mathbb R$ such that $\mathscr D_0f_\nu\in W^{s,2}\left(\mathbb R^{2n},\mathscr V_0\right),$ supp$\left(\mathscr D_0f_\nu\right)\subset K$ for all $\nu$ and $\mathscr D_0f_\nu\rightarrow g$ in $W^{s,2}\left(\mathbb R^{2n},\mathscr V_1\right)$ \cite{Schwartz}.
Then $\mathscr D_0f_\nu=0$ outside of $K,$ i.e. $f_\nu$ is  monogenic on $\mathbb R^{2n}\setminus K.$  We can assume that $\mathbb R^{2n}\setminus K$  has no compact connected component. By Proposition \ref{PV} each component of a  monogenic function annihilated by $\Delta$ and so it is biharmonic on $\mathbb R^{2n}\setminus K$ and real analytic. Consequently, as compacted supported distributions, monogenic  functions $\left.f_\nu\right|_{\mathbb R^{2n}\setminus K}$  vanish on  $\mathbb R^{2n}\setminus K$ and thus $\left\{f_\nu\right\}$ are also supported in $K.$
This argument also implies that  (\ref{esi}) holds with $C_{s,K} = 0.$
If this is not true, there exist a sequence $h_\nu\in W^{s,2}\left(\mathbb R^{2n}, \mathscr V_0\right),$ with  supp$h_\nu\subset K,$ such that
\begin{align*}
\left\|\mathscr D_0h_\nu\right\|_{W^{s,2}}^2<\frac1\nu\left\|h_\nu\right\|_{W^{s+1,2}}^2.
\end{align*}
By rescaling we can assume that  $\left\|h_\nu\right\|_{W^{s,2}}= 1$ for each $\nu.$ By (\ref{esi})
\begin{align*}
C_{s,K}\geq\left(c_{s,K}-\frac1\nu\right)\left\|h_\nu\right\|_{W^{s+1,2}}^2.
\end{align*}
Thus $\left\{h_\nu\right\}$ is bounded in the Sobolev space $W^{s+1,2}\left(\mathbb R^{2n},\mathscr V_0\right).$ By the well known compactness
of the inclusion $W^{s+1,2}\left(\mathbb R^{2n},\mathscr V_0\right)\subset W^{s,2}\left(\mathbb R^{2n},\mathscr V_0\right),$ there is a subsequence that converges to a function $h_\infty\in W^{s,2}\left(\mathbb R^{2n},\mathscr V_0\right).$  We have $\left\|h_\infty\right\|_{W^{s,2}}=1,\ \mathscr D_0h_\infty\equiv0.$
Then $\Delta^2 h_\infty=0$ and $h_\infty$ is also compactly supported in $K$. So $h_\infty=0$ by analytic continuation, which contradicts to $\left\|h_\infty\right\|_{W^{s,2}}^2=1.$

By the estimate (\ref{esi})   with $C_{s,K} = 0,$ we see that $\left\{f_\nu\right\}$ is uniformly
bounded in $W^{s+1,2}\left(\mathbb R^{2n},\mathscr V_0\right)$, and hence contains a subsequence which converges to a compactly supported weak solution $f\in W^{s,2}\left(\mathbb R^{2n},\mathscr V_0\right)$  of $\mathscr D_0f=g.$ Namely, the image of $\mathscr D_0$
in $\mathcal E'\left(\mathbb R^{2n},\mathscr V_0\right)$ is closed. The closedness of the image of $\mathscr D_0$ in $\mathcal D\left(\mathbb R^{2n},\mathscr V_0\right) $ follows from the proved result for $\mathcal E'\left(\mathbb R^{2n},\mathscr V_0\right)$ and the elliptic regularity.
\end{proof}
%See \cite[Theorem 7.3]{wang29} for this theorem on tangential $k$-Cauchy--Fueter operator over right group.
Then we can prove the   generalization of Malgrange's vanishing Theorem \ref{vanish}. See \cite{Brinkschulte,Laurent,Nacinovich} for Malgrange's vanishing theorem  on CR manifolds and    \cite[Theorem 1.3]{wang29} for this theorem for tangential $k$-Cauchy--Fueter operator over right group, respectively.

{\it Proof of Theorem \ref{vanish}.}
By Theorem \ref{th73} and its proof, the sequences
\begin{equation}\begin{aligned}\label{aco1}
&0\rightarrow\mathcal D\left(\mathbb R^{2n},\mathscr V_0\right)\xrightarrow{{\mathscr{D}}_{0}} \mathcal D\left(\mathbb R^{2n},  \mathscr V_1\right),\\&0\rightarrow\mathcal E'\left(\mathbb R^{2n}, \mathscr V_0\right)\xrightarrow{{\mathscr{D}}_{0}} \mathcal E'\left(\mathbb R^{2n}, \mathscr V_1\right),
\end{aligned}\end{equation}
are both exact and have closed ranges. Thus $\mathscr D_0$'s in (\ref{aco1}) are topological homomorphisms. We can apply abstract duality theorem \ref{dual} (6) to sequences in (\ref{aco1}) to get exact sequences
\begin{equation}\begin{aligned}\label{aco2}
&0\leftarrow\mathcal D'\left(\mathbb R^{2n},  {\mathscr V}_0\right)\xleftarrow{\widehat{\mathscr{D}}_{0}} \mathcal D'\left(\mathbb R^{2n}, {\mathscr V}_1\right),\\&0\leftarrow\mathcal E\left(\mathbb R^{2n},  {\mathscr V}_0\right)\xleftarrow{\widehat{\mathscr{D}}_{0}} \mathcal E\left(\mathbb R^{2n}, {\mathscr V}_1\right),
\end{aligned}\end{equation}
i.e. $\widehat{\mathscr D}_0$'s are surjective, since $\ker\mathscr D_0=\{0\}$ in (\ref{aco1}).
Compare the formula  (\ref{hat}) of $\widehat{\mathscr D}_0$ and the definition of $\mathscr D_2$ in (\ref{D2}), $\widehat{\mathscr D}_0:\mathcal E\left(\mathbb R^{2n}, {\mathscr V}_1\right)\rightarrow \mathcal E\left(\mathbb R^{2n}, {\mathscr V}_0\right)$ can be identified with ${\mathscr D}_2:\mathcal E\left(\mathbb R^{2n}, {\mathscr V}_2\right)\rightarrow \mathcal E\left(\mathbb R^{2n}, {\mathscr V}_3\right)$ if we take linear isomorphism    $\mathscr V_2\rightarrow \mathscr V_1$ given by $\left(\begin{matrix}  h_0\\h_1\end{matrix}\right)\rightarrow\left(\begin{matrix}  h_1\\-h_0\end{matrix}\right)$ and identify $\mathbb S^+$ with $\mathbb S^-$ as complex vector spaces, although they are different as representations of  ${\rm Spin}(n)$.
Then we have the exact sequences
\begin{equation*}\begin{aligned}
\mathcal E\left(\mathbb R^{2n}, \mathscr V_2\right)\xrightarrow{{\mathscr{D}}_{0}} \mathcal E\left(\mathbb R^{2n}, \mathscr V_3\right)\rightarrow 0.
\end{aligned}\end{equation*} Similarly, $\mathcal D'\left(\mathbb R^{2n},\mathscr V_2\right)\xrightarrow{{\mathscr{D}}_{0}} \mathcal D'\left(\mathbb R^{2n},  \mathscr V_3\right)\rightarrow0.$
 The theorem is proved.
\qed
\begin{rem}\label{remh}The exactness of (\ref{aco1}) is equivalent to $H^0\left(\mathcal D\left(\mathbb R^{2n},{\mathscr V}_\bullet\right)\right)=$ $H^0\left(\mathcal E'\left(\mathbb R^{2n},{\mathscr V}_\bullet\right)\right)=\{0\},$ while Theorem \ref{vanish}
 is equivalent to  $H_0\left(\mathcal D'\left(\mathbb R^{2n},{\mathscr V}_\bullet\right)\right)$ $=H_0\left(\mathcal E\left(\mathbb R^{2n},{\mathscr V}_\bullet\right)\right)=\{0\}.$
\end{rem}

%\begin{lem} {\rm (Stokes-type formula)} Let be a bounded domain on the group $\mathbb R^{2\times n}$ with smooth boundary and defining function let $\rho$ be a defining function (i.e. $\rho = 0$ on $\partial\Omega$	 and $\rho < 0$ in $\Omega$) such that $|{\rm grad}\rho| = 1.$ Then we have
%\begin{align}
%\int_{\Omega} h\nabla_Af{\rm d}V=-\int_{\Omega}f\nabla_Ah{\rm d}V+\int_{\partial \Omega}f\nabla_A\rho{\rm d}S,
%\end{align}
%for $A=0,1.$
%\end{lem}

Now we can prove the  Hartogs--Bochner extension Theorem \ref{text} for  monogenic functions.

{\it Proof of Theorem \ref{text}.}
Since $\mathscr D_0f$ vanishes to the second order on $\partial \Omega,$  we can  extend  $\mathscr D_0 f$ by $0$ outside of $\overline\Omega$ to get a $\mathscr D_1$-closed $C^2$ element $F\in\mathcal E'\left(\mathbb R^{2n},\mathscr V_1\right)$ supported in $\overline \Omega.$ Since $H^0\left(\mathcal E'\left(\mathbb R^{2n}, {\mathscr V}_\bullet\right)\right)$ vanish by Remark \ref{remh},  it is separated. Thus, we can apply abstract duality theorem  \ref{dual} (1) to the second sequences in (\ref{aco1}) and (\ref{aco2}) to see that
\begin{align}\label{imd}
{\rm Im}\ \mathscr D_0=\left\{\tilde F\in\mathcal E'\left(\mathbb R^{2n},\mathscr V_1\right)|\left\langle \tilde F,G\right\rangle=0\ {\rm for\ any}\ G\in\ker \widehat{\mathscr D}_0 \right\}.
\end{align}
Consequently, we have $F\in{\rm Im}\ \mathscr D_0.$ This is  because $\widehat{\mathscr D}_0=\mathscr D_0^*$ on $\mathcal E\left(\mathbb R^{2n},\mathscr V_1\right)$ and so for any $G\in\ker  {\mathscr D}^*_0\subset\mathcal E\left(\mathbb R^{2n}, {\mathscr V}_1\right),$ we have
\begin{equation}\begin{aligned}\label{47}
\langle F,G\rangle=&\int_\Omega\left\langle\mathscr D_0f,G\right\rangle_{\mathscr V_1}{\rm d}V=\sum_{A=0,1}\int_\Omega\left\langle\nabla_Af, G_A\right\rangle_{\mathbb S^-}{\rm d}V\\=&\sum_{A=0,1}\sum_{j}\int_\Omega\left\langle  \gamma_j\partial_{Aj}f, G_A\right\rangle_{\mathbb S^-}{\rm d}V\\=&-\sum_{A=0,1}\sum_{j}\int_\Omega\left\langle\partial_{Aj}f, \gamma_j G_A\right\rangle_{\mathbb S^+}{\rm d}V\\=&\sum_{A=0,1}\int_\Omega \left\langle f,\nabla_AG_A\right\rangle_{\mathbb S^+}{\rm d}V-\sum_{A=0,1}\sum_j\int_{\partial \Omega}\left\langle f,\gamma_jG_A\right\rangle_{\mathbb S^+}n_{Aj} {{\rm d}S} \\=&\left\langle f, {\mathscr D}^*_0G\right\rangle-\sum_{A=0,1}\sum_j\int_{\partial \Omega}\left\langle f,\gamma_jG_A\right\rangle_{\mathbb S^+}n_{Aj} {{\rm d}S}=0,
\end{aligned}\end{equation}
by using (\ref{hat}) and Stokes-type formula, where $n$ is the unit outer normal vector   to $\partial\Omega.$ Hence, by (\ref{imd}) there exists a distribution $H\in\mathcal E'\left(\mathbb R^{2n},\mathscr V_0\right)$   such that $F =\mathscr D_0H.$

Recall that a distribution in $\mathcal E'$ always has compact support. Now by estimate  (\ref{esi}) with $C_{s,K}=0$ and $s=0,$ proved in the proof of Theorem \ref{th73}, $H\in W^{1,2}\left(\mathbb R^{2n},\mathscr V_0\right).$   Then $H$ is  monogenic on the connected open set $\mathbb R^{2n}\setminus\overline \Omega,$ since supp $F\in\overline\Omega.$ By real analyticity of monogenic  functions, $H$ vanishes on $\mathbb R^{2n}\setminus\overline \Omega.$ Hence, $\tilde f = f-H$ gives us the required extension.
\qed
\begin{rem}\label{r4.7}
The conditions   $\mathscr D_0 f$ vanishing to the second order on $\partial \Omega$ and  the moment condition (\ref{moment}) are also  necessary for such extension in Theorem \ref{text}.
It is clear that if $f$ is monogenic, for any   $G\in\ker {\mathscr D}^*_0\subset\mathcal E\left(\mathbb R^{2n}, {\mathscr V}_1\right),$
\begin{equation*}\begin{aligned}
\sum_{A=0,1}\sum_j\int_{\partial \Omega}\left\langle f,\gamma_jG_A\right\rangle_{\mathbb S^+}n_{Aj} {{\rm d}S}=%&\sum_{A=0,1}\sum_j\int_{\partial \Omega}\left\langle  f,\gamma_jG_A\right\rangle_{\mathbb S^+}n_{Aj} {{\rm d}S}\\=&
\left\langle f, {\mathscr D}^*_0G\right\rangle-\left\langle\mathscr D_0 f,G\right\rangle=0,
\end{aligned}\end{equation*}as in (\ref{47}).
\end{rem}

\subsection*{Acknowledgment}
The first author is  partially supported by  Nature Science Foundation of Zhejiang province (No. LY22A010013) and  National Nature Science Foundation in China (Nos. 11801508, 11971425); The second author is partially supported by National Nature Science Foundation in China  (No. 11971425); The third author is partially supported by National Nature Science Foundation in China  (No. 12071197), the Natural Science
Foundation of Shandong Province (Nos. ZR2019YQ04, 2020KJI002).

% ------------------------------------------------------------------------
\end{document}